\pdfoutput=1 
\documentclass[a4paper, bibliography=totoc]{scrartcl}

\usepackage{orcidlink}

\usepackage{amsmath}
\usepackage{amssymb}
\usepackage{enumerate}
\usepackage{amsthm}
\usepackage{mathtools} 
\usepackage{stackrel} 

\usepackage[UKenglish]{babel}
\usepackage[T1]{fontenc}
\usepackage[utf8]{inputenc}

\usepackage{thmtools} 
\usepackage{thm-restate} 

\usepackage{hyperref} 
\definecolor{blue3}{rgb}{.1,.0,.4}
\hypersetup{colorlinks=true, linkcolor=red, urlcolor=blue3, citecolor=blue3}
\hypersetup{pdftitle=Divergence of separated nets with respect to displacement equivalence}
\providecommand{\keywords}[1]{\textbf{\textit{Keywords:}} #1}
\providecommand{\subclass}[1]{\textbf{\textit{MSC classes:}} #1}

\declaretheorem[name=Theorem,numberwithin=section]{thm} 
\newtheorem{define}[thm]{Definition}
\newtheorem{lemma}[define]{Lemma}
\newtheorem{construction}[define]{Construction}
\newtheorem{prop}[define]{Proposition}

\newtheorem{remark}[define]{Remark}
\newtheorem{cor}[define]{Corollary}
\newtheorem{conjecture}[define]{Conjecture}
\newtheorem{example}[define]{Example}

\newcommand{\R}{\mathbb{R}}

\newcommand{\Z}{\mathbb{Z}}
\newcommand{\N}{\mathbb{N}}

\newcommand{\lip}{\operatorname{Lip}}
\newcommand{\dist}{\operatorname{dist}}
\newcommand{\id}{\operatorname{id}}
\newcommand{\leb}{\mathcal{L}}

\DeclareMathOperator{\disp}{disp}
\DeclareMathOperator{\diam}{diam}

\newcommand{\mb}[1]{\mathbf{#1}}

\newcommand{\mc}[1]{\mathcal{#1}}

\newcommand{\niceint}[3]{\int_{#1}{#2\,\mathrm{d}#3}}

\newcommand{\abs}[1]{\left|#1\right|}
\newcommand{\lnorm}[2]{\left\|#2\right\|_#1}
\newcommand{\enorm}[1]{\lnorm{2}{#1}}

\newcommand{\set}[1]{\left\{#1\right\}}
\newcommand{\cl}[1]{\overline{#1}}
\newcommand{\rest}[2]{#1|_{#2}}
\DeclarePairedDelimiter{\floor}{\lfloor}{\rfloor}
\DeclarePairedDelimiter{\br}{(}{)}

\title{Divergence of separated nets with respect to displacement equivalence\thanks{The authors acknowledge the support of Austrian Science Fund (FWF): P 30902-N35. The work was started when both authors were employed at University of Innsbruck and continued when the first named author was employed at University of Leipzig and the second named author was employed at Institute of Science and Technology of Austria, where he was supported by an IST Fellowship.}}
\author{
	Michael Dymond\,\orcidlink{0000-0002-1900-3549}
	\\ \small{School of Mathematics}
	\\ \small{Watson Building}
	\\ \small{University of Birmingham}
	\\ \small{Edgbaston, Birmingham B15 2TT, UK}
	\\ \small{\href{m.dymond@bham.ac.uk}{m.dymond@bham.ac.uk}}
	\and
	Vojt\v ech Kalu\v za\,\orcidlink{0000-0002-2512-8698}
	\\ \small{Institute of Science and Technology Austria}
	\\ \small{Am Campus 1}
	\\ \small{3400 Klosterneuburg, Austria}
	\\ \small{\href{mailto:vojtech.kaluza@ist.ac.at}{vojtech.kaluza@ist.ac.at}}
}
\date{}

\begin{document}
	\maketitle
	\abstract{We introduce a hierarchy of equivalence relations on the set of separated nets of a given Euclidean space, indexed by concave increasing functions $\phi\colon (0,\infty)\to(0,\infty)$. Two separated nets are called \emph{$\phi$-displacement equivalent} if, roughly speaking, there is a bijection between them which, for large radii $R$, displaces points of norm at most $R$ by something of order at most $\phi(R)$. We show that the spectrum of $\phi$-displacement equivalence spans from the established notion of \emph{bounded displacement equivalence}, which corresponds to bounded $\phi$, to the indiscrete equivalence relation, corresponding to $\phi(R)\in \Omega(R)$, in which all separated nets are equivalent. In between the two ends of this spectrum, the notions of $\phi$-displacement equivalence are shown to be pairwise distinct with respect to the asymptotic classes of $\phi(R)$ for $R\to\infty$. We further undertake a comparison of our notion of $\phi$-displacement equivalence with previously studied relations on separated nets. Particular attention is given to the interaction of the notions of $\phi$-displacement equivalence with that of \emph{bilipschitz equivalence}.}
	
	\medskip
	
	\subclass{51F99, 51M05, 52C99, 26B35, 26B10}

	\section{Introduction}
In the present work, we compare the metric structures of separated nets by examining how much mappings between them displace points. The notion of displacement of a mapping is defined as follows:

\begin{define}
	Let $f\colon A\subseteq\R^d\to\R^d$. We define the \emph{displacement constant} of $f$ as
	\[
	\disp(f):=\lnorm{\infty}{f-\id}.
	\]
	If $\disp(f)<\infty$, then we say that $f$ is a mapping of \emph{bounded displacement}. 
\end{define}

Research into separated nets in Euclidean spaces has broadly centred around the question of to what extent any two separated nets in a Euclidean space are similar, as metric spaces. To formulate this question more precisely, it is necessary to prescribe what it means for two separated nets to be considered similar, or put differently, to define a symmetric relation on the class of separated nets in a Euclidean space. Two such notions, which are in fact equivalence relations, have been studied most prominently. 

The most narrow of these notions is that of bounded displacement equivalence. Two separated nets $X,Y\subseteq \R^{d}$ are said to be \emph{bounded displacement equivalent}, or \emph{BD equivalent}, if there exists a bijection $f\colon X\to Y$ for which $\disp(f)<\infty$. To demonstrate how constrictive BD equivalence is, we point out that for any separated net $X\subseteq \R^{d}$, $X$ is not $BD$ equivalent to $2X$.\footnote{The reader may wish to verify this as an exercise; alternatively we note that this fact is a special case of Proposition~\ref{prop:nat_dens_different} of the present work.} Hence, even linear bijections $\R^{d}\to\R^{d}$ may transform a separated net to a BD non-equivalent separated net.

For the second notion, two separated nets $X,Y\subseteq \R^{d}$ are called \emph{bilipschitz equivalent}, or \emph{BL equivalent}, if there is a bilipschitz bijection $f\colon X\to Y$. This defines a much looser form of equivalence in comparison to BD equivalence.  In fact, it is a highly non-trivial question, posed by Gromov~\cite{Grom} in 1993, whether BL equivalence distinguishes at all between the separated nets of a  multidimensional Euclidean space. Moreover, we point out that BD equivalence is easily seen to be stronger than BL equivalence.

For all Euclidean spaces of dimension at least two, Gromov's question was answered negatively in 1998 by Burago and Kleiner~\cite{BK1} and (independently) McMullen~\cite{McM}; the papers \cite{BK1} and \cite{McM} verify the existence of a separated nets in $\R^{d}$, $d\geq 2$, which are not BL equivalent to the integer lattice. 

In the recent work \cite{dymond_kaluza2019highly}, the authors introduce the notion of $\omega$-regularity of a separated net. 

\begin{define}\label{def:omeg_reg}
Given separated nets $X,Y\subseteq \R^{d}$ and a strictly increasing, concave function $\omega$ defined on a positive open interval starting at $0$ and satisfying $\lim_{t\to 0}\omega(t)=0$, a mapping $f\colon X\to Y$ is called a \emph{homogeneous $\omega$-mapping} if there are constants $K>1$ and $a>0$ such that
\[
\lnorm{2}{f(y)-f(x)}\leq KR\omega\left(\frac{\lnorm{2}{y-x}}{R}\right)
\]
for all $R>0$ and $x,y\in X\cap \overline{B}(0,R)$ with $\lnorm{2}{y-x}<aR$. The separated net $X\subseteq \R^{d}$ is called \emph{$\omega$-regular with respect to the separated net $Y\subseteq \R^{d}$} if there exists a bijection $f\colon X\to Y$ such that both $f$ and $f^{-1}$ are homogeneous $\omega$-mappings. Otherwise $X$ is called \emph{$\omega$-irregular with respect to $Y$}. In the case that $Y=\Z^{d}$, these terms are shortened to \emph{$\omega$-regular} and \emph{$\omega$-irregular} respectively.
\end{define}
From now on, we will refer to functions $\omega$ with the properties given in Definition~\ref{def:omeg_reg} as \emph{moduli of continuity}. The function $\omega(t)=t$ will be called the \emph{Lipschitz modulus of continuity} and functions $\omega(t)=t^{\beta}$ with $\beta\in(0,1)$ will be referred to as \emph{H\"older moduli of continuity}. When we prescribe a modulus of continuity $\omega$ by a formula such as $\omega(t)=t^{\beta}$, it should be understood that this formula defines $\omega$ on some interval $(0,a)$ with $a>0$. The precise value of $a$ and indeed the behaviour of $\omega(t)$ for $t\geq a$ is irrelevant to the notions of Definition~\ref{def:omeg_reg}.

It is clear that for two moduli of continuity $\omega_{1}$, $\omega_{2}$ satisfying $\omega_{2}(t)\in o(\omega_{1}(t))$ for $t\to 0$,\footnote{We use the standard asymptotic notation $O, o, \Omega$ and $\Theta$; for the definitions, see Section~\ref{sec:prelim}.}
the notion of $\omega_{1}$-regularity is formally weaker than that of $\omega_{2}$-regularity. Further for the Lipschitz modulus of continuity $\omega(t)=t$, $\omega$-regularity of $X$ with respect to $Y$ is nothing other than the BL equivalence of $X$ and $Y$. Thus, the result of Burago and Kleiner and (independently) McMullen discussed above can be formulated as follows: In every Euclidean space $\R^{d}$ with $d\geq 2$ there exists an $\omega$-irregular separated net for the function $\omega(t)=t$.

The notion of $\omega$-regularity of separated nets is motivated by a result of McMullen~\cite[Theorem~5.1]{McM}, which stands in contrast to the existence of BL non-equivalent nets. McMullen~\cite{McM} proves that for any two separated nets $X$ and $Y$ in Euclidean space, $X$ is $\omega$-regular with respect to $Y$ for some Hölder modulus of continuity $\omega(t)=t^{\beta}$ for some $\beta\in (0,1)$. 
In the work \cite{dymond_kaluza2019highly}, the present authors investigate $\omega$-regularity for $\omega$ lying asymptotically in between the Lipschitz modulus of continuity and Hölder moduli of continuity. The paper \cite{dymond_kaluza2019highly} proves that there are separated nets in every $\R^{d}$, $d\geq 2$, which are $\omega$-irregular for the modulus of continuity 
\begin{equation}\label{eq:our_omega}
\omega(t)=t\left(\log\frac{1}{t}\right)^{\alpha_{0}},
\end{equation}
where $\alpha_{0}=\alpha_{0}(d)$ is a positive constant determined by the dimension $d$ of the space. This is formally a stronger result than the existence of BL non-equivalent separated nets.

\subsection*{Growth of restricted displacement constants.}
 Looking at the value of $\disp(f)$ for bijections $f$ between two separated nets $X$ and $Y$ gives only a very crude comparison of their metric structures. Roughly speaking `most' pairs of separated nets $X$ and $Y$ are BD non-equivalent, so that $\disp(f)=\infty$ for every such bijection. This motivates a more subtle form of metric comparison of separated nets in Euclidean space via displacement:
 
\begin{define}\label{def:dispR}
	Let $f\colon A\subseteq\R^{d}\to \R^{d}$. We define a function $(0,\infty)\to [0,\infty)$ by 
	\[
	R\mapsto\disp_R(f):=\begin{cases}
	\disp(\rest{f}{A\cap\cl{B}(\mb{0},R)})& \text{ if $A\cap \cl{B}(\mb{0},R)\neq \emptyset$},\\
	0 & \text{otherwise}.
	\end{cases}
	\]
\end{define}

Although we expect generally that $\disp(f)=\infty$ for any bijection between two separated nets, so that $\lim_{R\to \infty}\disp_{R}(f)=\infty$, it remains of interest in such cases to determine the optimal asymptotic growth of $\disp_{R}(f)$ as $R\to \infty$ among such bijections. Indeed, this allows for a more flexible notion of displacement equivalence. 

\begin{define}\label{def:disp_eq}
	Let $\phi\colon (0,\infty)\to (0,\infty)$ be an increasing, concave function and $X$ and $Y$ be separated nets of $\R^{d}$. We say that $X$ and $Y$ are $\phi$-displacement equivalent if there exists a bijection $f\colon X\to Y$ for which $\disp_{R}(f)\in O(\phi(R))$. 
\end{define}

\begin{remark}
	In Definition~\ref{def:dispR} and Definition~\ref{def:disp_eq} it may appear that the origin $\mb{0}\in\R^{d}$ has a special role: it is the reference point with respect to which the quantity $\disp_{R}(f)$ is defined. It is therefore natural to ask, whether a different choice of reference point in Definition~\ref{def:dispR} would give rise to a different notion of $\phi$-displacement equivalence in Definition~\ref{def:disp_eq}. However, this is not the case, due to the conditions on the functions $\phi$ admitted in Definition~\ref{def:disp_eq} and the inequality
	\[
	\disp_{R}^{y}(f)\leq \disp_{R+\enorm{z-y}}^{z}(f),
	\] 
	where $\disp_{R}^{w}(f)$ denotes the quantity of Definition~\ref{def:dispR} obtained when $w\in\R^{d}$ is used as the reference point instead of $\mb{0}\in\R^{d}$.
\end{remark}

\begin{remark}\label{rem:concave}
	We require the concavity of $\phi$ in Definition~\ref{def:disp_eq} in order to verify that $\phi$-displacement equivalence is a true equivalence relation. However, the reader may ask whether it is possible to admit a larger class of functions $\phi$.
	It is the authors' view that admitting only concave functions $\phi$ in Definition~\ref{def:disp_eq} is not a major restriction. Recall that for every increasing function $\psi\colon (0,\infty)\to (0,\infty)$ with $\psi\in O(R)$ there is a concave majorant, that is, a concave increasing function $\phi\colon (0,\infty)\to (0,\infty)$ such that $\psi\leq \phi$ pointwise and $\psi(R)\notin o(\phi(R))$; see Lemma~\ref{lemma:concave_maj}. 
		
Observe that the concave condition in Definition~\ref{def:disp_eq} implies that $\phi(R)\in O(R)$ and thus superlinear functions such as $\phi(R)=R^{2}$ are excluded. However, excluding superlinear functions $\phi$ is not any restriction because, were they to be admitted, then the resulting notions of $\phi(R)$-displacement equivalence for all functions $\phi(R)\in\Omega(R)$ would coincide and equal the trivial equivalence relation in which all separated nets of $\R^{d}$ are equivalent. This last assertion is a consequence of Proposition~\ref{p:disp_always_linear} of the present work.
\end{remark}

\subsection*{Structure of the Paper and Main Results} To finish this introduction, we outline the structure of the paper and summarise the main contributions of the present work. 

Section~2 and 3 present preliminary results and observations which can mostly be thought of as easy consequences of the new definition of $\phi$-displace\-ment equivalence, but are nevertheless worth highlighting in view of the authors. In Section 2 we
verify that the notions of $\phi$-displacement equivalence given by Definition~\ref{def:disp_eq} are equivalence relations:
\begin{restatable*}{prop}{equivrel}\label{prop:equivrel}
		Let $\phi\colon (0,\infty)\to(0,\infty)$ be an increasing, concave function. Then the notion of $\phi$-displacement equivalence of separated nets in $\R^{d}$ given by Definition~\ref{def:disp_eq} is an equivalence relation on the set of separated nets of $\R^{d}$.
\end{restatable*}

We further show that the notion of $\phi$-displacement equivalence for $\phi(R)\in \Omega(R)$ does not distinguish between separated nets:
\begin{restatable*}{prop}{linear}\label{p:disp_always_linear}
	Let $X,Y$ be two separated nets in $\R^d$. Then there is a bijection $f\colon X\to Y$ such that $\disp_R(f),\disp_R(f^{-1})\in O(R)$.
\end{restatable*}
In contrast, Section 3 deals with negative results and identifies certain barriers to $\phi$-displacement equivalence for $\phi\in o(R)$. 

Our first main result demonstrates that the notions of $\phi$-displacement equivalence for increasing, concave functions $\phi\colon (0,\infty)\to(0,\infty)$ form a fine spectrum starting from the strictest form of $\phi$-displacement equivalence, namely BD equivalence, which corresponds to $\phi$-equivalence for bounded $\phi(R)\in O(1)$, to the weakest form of $\phi$-displacement equivalence, namely that corresponding to $\phi(R)\in\Omega(R)$. In the spectrum between $O(1)$ and $\Omega(R)$ we show that the notions of $\phi$-displacement equivalence are pairwise distinct with respect to the asymptotic classes of functions $\phi(R)$ for $R\to\infty$. We prove namely the following statement:
\begin{restatable*}{thm}{dispcls}\label{thm:displacement_class}
	Let $\phi\colon (0,\infty)\to (0,\infty)$ be an increasing, concave function with $\phi(R)\in o(R)$ and $X\subseteq \R^{d}$ be a separated net. Then there exists a separated net $Y\subseteq \R^{d}$ such that every bijection $f\colon X\to Y$ satisfies $\disp_{R}(f)\notin  o(\phi(R))$ and there exists a bijection $g\colon X\to Y$ with $\disp_{R}(g),\disp_{R}(g^{-1})\in O(\phi(R))$. Moreover, such $Y$ can be found so that $X$ and $Y$ are bilipschitz equivalent.
\end{restatable*}
\begin{cor}\label{cor:dist_disp_eq}
	Let $\phi_{1},\phi_{2}\colon (0,\infty)\to (0,\infty)$ be increasing, concave functions with $\phi_{1}(R)\in o(\phi_{2}(R))$. Then $\phi_{2}$-displacement equivalence of separated nets in $\R^{d}$ is a strictly weaker notion than that of $\phi_{1}$-displacement equivalence.
\end{cor}

Theorem~\ref{thm:displacement_class} will be proved in Section~\ref{sec:disp_class}; Corollary~\ref{cor:dist_disp_eq} is an immediate consequence of Theorem~\ref{thm:displacement_class}. Note that Theorem~\ref{thm:displacement_class} also verifies the optimality of Proposition~\ref{p:disp_always_linear}.

The theme of Sections~\ref{sec:BD_class} and \ref{sec:BL_inter_disp} is the comparison of the established notion of BL equivalence with the spectrum of $\phi$-displacement equivalence for increasing, concave functions $\phi\colon (0,\infty)\to(0,\infty)$. We begin, in Section~\ref{sec:BD_class}, with the strictest form of $\phi$-displacement equivalence, namely BD equivalence. In Section~\ref{sec:BL_inter_disp} we then move onto $\phi$-displacement equivalence for unbounded $\phi$.

We compare the notions of BL equivalence and $\phi$-displacement equivalence by looking at the intersection of the BL equivalence classes with the classes of $\phi$-displacement equivalence. The cardinality of the set of equivalence classes of separated nets has already attracted some research attention. Magazinov~\cite{Magazinov} shows that in every Euclidean space of dimension at least two, the set of BL equivalence classes of separated nets has the cardinality of the continuum. Since BD equivalence is stronger than BL equivalence, this also implies that there are uncountably many distinct BD classes. In \cite[Theorem~1.3]{Frettloeh_et_al-BD2019}, Frettlöh, Smilansky and Solomon also verify the existence of uncountably many, pairwise distinct BD equivalence classes of separated nets in $\R^{2}$. Interestingly, the class representatives of the uncountably many, pairwise distinct BD equivalence classes given in \cite{Frettloeh_et_al-BD2019} all come from the same BL equivalence class.

Independently of the aforementioned works \cite{Magazinov} and \cite{Frettloeh_et_al-BD2019}, we are able to verify that every Euclidean space has uncountably many, pairwise distinct BD equivalence classes of separated nets. Further, we provide new information, namely that there are uncountably many pairwise distinct BD equivalence classes inside each BL equivalence class. Hence, we are able to present a new result, which we prove in Section~\ref{sec:BD_class}:
\begin{restatable*}{thm}{BDclasses}\label{thm:BD_equiv_classes}
For every $d\in \N$, every bilipschitz equivalence class of separated nets in $\R^{d}$ decomposes as a union of uncountably many pairwise distinct bounded displacement equivalence classes.	
\end{restatable*}
For unbounded functions $\phi(R)$, the analysis of the interaction of the BL classes and the $\phi$-displacement equivalence classes of separated nets in $\R^{d}$ is more challenging. In light of Theorem~\ref{thm:BD_equiv_classes}, the natural problem is to characterise the increasing, concave functions $\phi(R)\in o(R)$ for which $\phi$-displacement equivalence is stronger than BL equivalence; note that Theorem~\ref{thm:BD_equiv_classes} takes care of the functions $\phi(R)\in O(1)$. In Section~\ref{sec:BL_inter_disp} we resolve this matter. We verify, namely, that $\phi$-displacement is stronger than BL equivalence if and only if $\phi(R)\in O(1)$. In particular, this means that BD equivalence is the only form of $\phi$-displacement equivalence for which Theorem~\ref{thm:BD_equiv_classes} holds.

Section~\ref{sec:BL_inter_disp} should also be placed in the context of $\omega$-regularity of separated nets. Recall that BL equivalence corresponds to the notion of $\omega$-regularity for the modulus of continuity $\omega(t)=t$. For the weaker modulus of continuity $\omega$ of \eqref{eq:our_omega} and any function $\phi(R)\in O\left(R\omega\left(\frac{1}{R}\right)\right)$,
the authors prove in \cite{dymond_kaluza2019highly} that the $\phi$-displacement equivalence class of the integer lattice does not contain any $\omega$-irregular separated nets. This may support the following conjecture:
\begin{conjecture}\label{conject:omega_phi}
	Let $d\geq 2$, $\omega$ be a modulus of continuity in the sense of Definition~\ref{def:omeg_reg} and $\phi\colon(0,\infty)\to(0,\infty)$ be an increasing concave function. Then the class of $\omega$-irregular separated nets in $\R^{d}$ has non-empty intersection with the $\phi$-displacement equivalence class of the integer lattice $\Z^{d}$ if and only if $R\omega\left(\frac{1}{R}\right)\in o(\phi(R))$.
\end{conjecture}
Indeed the `only if' implication of Conjecture~\ref{conject:omega_phi} for the modulus of continuity $\omega$ of \eqref{eq:our_omega} is precisely the result \cite[Proposition~1.3]{dymond_kaluza2019highly} referred to above. In Section~\ref{sec:BL_inter_disp} of the present work, we show that for every increasing, unbounded, concave function $\phi\colon (0,\infty)\to(0,\infty)$, the $\phi$-displacement class of the integer lattice intersects distinct BL classes; in particular it contains $\omega$-irregular separated nets for $\omega(t)=t$. A matter of interest is whether every such $\phi$-displacement equivalence class intersects every BL equivalence class. This question remains open, but we are able to show that every such $\phi$-displacement equivalence class intersects uncountably many BL equivalence classes:
\begin{restatable*}{thm}{slowdisp}\label{thm:bil_neq_slow_disp}
	Let $d\geq 2$ and $\phi\colon (0,\infty)\to(0,\infty)$ be an unbounded, increasing, concave function. Then there is an uncountable family $(X_{\psi})_{\psi\in\Lambda}$ of pairwise bilipschitz non-equivalent separated nets in $\R^{d}$ for which each $X_{\psi}$ is $\phi$-displacement equivalent to $\Z^{d}$. 
\end{restatable*}
We point out that Theorem~\ref{thm:bil_neq_slow_disp} is a refinement of the lower bound from \cite{Magazinov} and is obtained entirely independently.
Moreover, put together with the fact that BD equivalence is stronger than BL equivalence, Theorem~\ref{thm:bil_neq_slow_disp} verifies Conjecture~\ref{conject:omega_phi} for the special case of the Lipschitz modulus of continuity $\omega(t)=t$. We further remark that equivalence of separated nets and cardinality of sets of their equivalence classes have been studied in connection with the notion of repetitivity of separated nets; see \cite{CortezNavas} and \cite{smilansky_solomon_2022}. Our methods for constructing separated nets verifying Theorem~\ref{thm:bil_neq_slow_disp} appear to destroy repetitivity and so there seem to be serious obstructions to employing them to, for example, prove a BL equivalence version of the dichotomy~\cite[Theorem~1.1]{smilansky_solomon_2022} for BD equivalence.

At this point, we also wish to state formally the characterisation announced in the above discussion of Section~\ref{sec:BL_inter_disp}. This result is an immediate consequence of Theorem~\ref{thm:BD_equiv_classes} and Theorem~\ref{thm:bil_neq_slow_disp}:
\begin{restatable}{thm}{phidispBL}\label{thm:phidispBL}
	Let $d\geq 2$ and $\phi\colon (0,\infty)\to (0,\infty)$ be an increasing, concave function. Then, $\phi$-displacement equivalence of separated nets in $\R^{d}$ is stronger than bilipschitz equivalence if and only if $\phi$ is bounded.
\end{restatable}

Finally, we finish this article in Section~\ref{sec:omeg_reg} with a useful application of the $\phi$-displacement equivalence spectrum. Whilst \cite{dymond_kaluza2019highly} verifies the existence of separated nets which are $\omega$-irregular for $\omega$ of the form \eqref{eq:our_omega}, it leaves one important issue unresolved: namely, whether $\omega$-regularity for $\omega$ of the form \eqref{eq:our_omega} is distinct from the notion of bilipschitz equivalence (that is, $\omega$-regularity for $\omega(t)=t$). In view of the results \cite[Theorem~5.1]{McM} and \cite[Theorem~1.1]{BK1}, it is clear that there are H\"older moduli of continuity of the form $\omega_{1}(t)=t^{\beta}$ for some $\beta\in (0,1)$ so that for $\omega_{2}(t)=t$, the notions of $\omega_{1}$- and $\omega_{2}$-regularity are distinct; $\omega_{1}$-regularity is strictly weaker than $\omega_{2}$-regularity. However, the most that can be established on the basis of the existing literature is that there are at least two distinct notions of $\omega$-regularity. In particular, \cite{dymond_kaluza2019highly} does not address the issue of whether there are any moduli of continuity $\omega$ strictly in between the H\"older moduli of continuity and the Lipschitz modulus of continuity, such as $\omega$ of the form \eqref{eq:our_omega}, which define further distinct notions of $\omega$-regularity. This is quite unsatisfactory because it leaves open the possibility that the result \cite[Theorem~1.2]{dymond_kaluza2019highly} is in fact identical to \cite[Theorem~1.1]{BK1} and the corresponding result in \cite{McM}, although it is formally stronger.

In the present article we verify that for the modulus of continuity $\omega$ of the form \eqref{eq:our_omega}, the notion of $\omega$-regularity is strictly weaker than BL equivalence. This confirms that the `highly irregular'\footnote{As asserted by the title of the work \cite{dymond_kaluza2019highly}.} separated nets given in \cite[Theorem~1.2]{dymond_kaluza2019highly} are indeed more irregular in a meaningful way than the BL non-equivalent separated nets of McMullen~\cite{McM} and Burago and Kleiner~\cite[Theorem~1.1]{BK1}. 
\begin{restatable*}{thm}{distinct}\label{thm:distinct}
	Let $d\geq 2$, $\alpha_{0}=\alpha_{0}(d)$ be the quantity of \cite[Theorem~1.2]{dymond_kaluza2019highly} and $\omega$ be a modulus of continuity in the sense of Definition~\ref{def:omeg_reg} such that $\omega(t)=t\left(\log \frac{1}{t}\right)^{\alpha_{0}}$ for $t\in (0,a)$ and some $a>0$. Then the set of $\omega$-regular  separated nets in $\R^{d}$ strictly contains the set of separated nets bilipschitz equivalent to $\Z^d$.
\end{restatable*}
Despite this progress, we are only able to increase the number of known pairwise distinct forms of $\omega$-regularity by one:
\begin{restatable*}{cor}{threedistinct}\label{cor:3distinct}
	For any dimension $d\geq 2$ 
	there exist moduli of continuity $\omega_{1},\omega_{2},\omega_{3}$ in the sense of Definition~\ref{def:omeg_reg} so that whenever $i,j\in\set{1,2,3}$ with $i<j$ the set of $\omega_{j}$-regular separated nets of $\R^{d}$ is strictly contained in the set of $\omega_{i}$-regular separated nets.
\end{restatable*}
It therefore remains an interesting research objective to expose the hierarchy of notions of $\omega$-regularity. The authors would conjecture that, at least for moduli of continuity $\omega$ lying asymptotically in between the Lipschitz modulus of continuity and the modulus of continuity of \eqref{eq:our_omega}, we get a fine hierarchy of notions of $\omega$-regularity. More precisely, we conjecture that whenever two moduli of continuity $\omega_{1}$ and $\omega_{2}$ satisfy $\omega_{2}\in o(\omega_{1}(t))$, $\omega_{2}(t)\in \Omega(t)$ and $\omega_{1}(t)\in O\left(t\log\left(\frac{1}{t}\right)^{\alpha_{0}}\right)$ for $t\to 0$, then the notion of $\omega_{1}$-regularity of separated nets in $\R^{d}$ is strictly weaker than that of $\omega_{2}$-regularity.

	\section{Preliminaries and Notation.}\label{sec:prelim}
\paragraph{Functions and Asymptotics.} Throughout the work we use the standard \emph{asymptotic notation} $O, o, \Omega,\Theta,$ with the following meaning. Let $f,g$ be two positive real-valued functions defined on an unbounded domain in $(0,\infty)$. For example, this allows for $f$ and $g$ to be real sequences. Then we write
\begin{align*}
f(x)\in O(g(x)) &\Longleftrightarrow \limsup_{x\to\infty}\frac{f(x)}{g(x)}<\infty,\\
f(x)\in o(g(x)) &\Longleftrightarrow \limsup_{x\to\infty}\frac{f(x)}{g(x)}=0,\\
f(x)\in \Omega(g(x)) &\Longleftrightarrow g(x)\in O(f(x)),\\
f(x)\in \Theta(g(x)) &\Longleftrightarrow f(x)\in O(g(x)) \text{ and } f(x)\in \Omega(g(x)).
\end{align*}
We sometimes write equations or inequalities using the above asymptotic notation. For example, the inequalities $c_{n}\leq n^{2}+O(n)\leq O(n^{2})$ should be interpreted as follows: there exist sequences $a_{n}\in O(n)$ and $b_{n}\in O(n^{2})$ such that $c_{n}\leq n^{2}+a_{n}\leq b_{n}$. Although the symbol $\omega$ also belongs to the standard asymptotic notation, we will avoid using it in this context. The reason for this is that we use the letter $\omega$ to denote moduli of continuity and for the notions of $\omega$-regularity of Definition~\ref{def:omeg_reg}. Since any asymptotic statement using the asymptotic $\omega$ notation can be rephrased using the little $o$ notation, this is not a problem.

A function $f\colon A\subseteq \R\to\R$ will be called \emph{increasing} if $f(t)\geq f(s)$ whenever $s,t\in A$ and $t\geq s$. If both inequalities $\geq$ in this condition may be replaced by the strict inequality $>$, then we call $f$ \emph{strictly increasing}. The notions of \emph{decreasing} and \emph{strictly decreasing} are defined analogously. 

We will require the following basic fact relating to concave majorants:
\begin{lemma}\label{lemma:concave_maj}
	Let $\psi\colon (0,\infty)\to (0,\infty)$ be an increasing function, $\phi\colon (0,\infty)\to (0,\infty)$ be a concave increasing function and suppose that $\psi(R)\in o(\phi(R))$. Then there exists a concave increasing function $\Psi\colon (0,\infty)\to (0,\infty)$ with the following properties:
	\begin{enumerate}[(a)]
		\item\label{maj1} $\psi(t)\leq \Psi(t)$ for all $t\in (0,\infty)$.
		\item\label{maj2} $\psi(R)\notin o(\Psi(R))$.
		\item\label{maj3} $\Psi(R)\in o(\phi(R))$.
	\end{enumerate}
\end{lemma}
\begin{proof}
	Consider the family $\mc{M}$ of all concave functions $\zeta\colon (0,\infty)\to (0,\infty)$ such that $\psi(t)\leq \zeta(t)$ for all $t\in (0,\infty)$. We define $\Psi\colon (0,\infty)\to(0,\infty)$ by
	\[
		\Psi(t)=\inf\set{\zeta(t)\colon \zeta\in \mc{M}}.
	\]
	As the pointwise infimum of a family of concave functions, $\Psi$ is itself concave. Moreover, the definitions of $\mc{M}$ and $\Psi$ ensure that \eqref{maj1} is satisfied. The concavity of $\Psi$, \eqref{maj1} and the fact that $\psi$ is increasing then imply that $\Psi$ is also increasing. To verify \eqref{maj2}, note first that boundedness of $\psi$ implies boundedness of $\Psi$. We may therefore assume that $\psi$ is unbounded. Let $\theta\in (0,1)$, $n\in\N$ and observe that the concave function $t\mapsto \theta \Psi(t)+\psi(n)$ does not belong to $\mc{M}$. We deduce from this the existence of $R_{n}\geq n$ such that $\theta\Psi(R_{n})\leq \psi(R_{n})-\psi(n)\leq \psi(R_{n})$. The sequence $(R_{n})_{n\in\N}$ obtained in this manner witnesses \eqref{maj2}. Finally, we prove \eqref{maj3}. Given $\varepsilon>0$, choose $T>0$ large enough so that $\frac{\psi(t)}{\phi(t)}\leq \varepsilon$ for all $t\geq T$. Then the function $t\mapsto \psi(T)+\varepsilon\phi(t)$ belongs to $\mc{M}$ and so
	\[
		\frac{\Psi(t)}{\phi(t)}\leq \frac{\psi(T)+\varepsilon\phi(t)}{\phi(t)}\leq 2\varepsilon
	\]
	for all $t\geq T$.
\end{proof}

\paragraph{Metric notions.} In a metric space $(M,\dist_{M})$, a set $Z\subseteq M$ will be called \emph{separated} if 
\[
\inf\set{\dist_{M}(z,z')\colon z,z'\in Z, z\neq z'}>0,
\]
and this infimum will be referred to as the \emph{separation constant} (or just the \emph{separation}) of $Z$ (in $M$). Moreover, $Z$ will be called \emph{$\delta$-separated} if its separation constant is at least $\delta$. We will refer to the set $Z$ as a \emph{net} of $M$ if 
\[
\sup\set{\inf_{z\in Z}\dist_{M}(z,x)\colon x\in M}<\infty,
\]
and this supremum will be called the \emph{net constant} of $Z$ in $M$. We will call $Z$ a \emph{$\theta$-net} of $M$ if its net constant is at most $\theta$.

Thus, $Z$ will be called a \emph{separated net} of (or in) $M$ if $Z$ is both separated and a net of $M$. Throughout the work, we will only be concerned with separated nets of subsets of a Euclidean space $\R^{d}$. For a set $F\subseteq \R^{d}$ the separated nets of $F$ are defined according to the above discussion, where the relevant metric space $M$ is given by the set $F$ together with the metric on $F$ induced by the Euclidean distance in $\R^{d}$.

Given two sets $S,T\subseteq \R^{d}$ we let 
\[
\dist(S,T):=\inf\set{\enorm{t-s}\colon s\in S,\,t\in T}.
\]
In the case that $S=\set{s}$ is a singleton we just write $\dist(s,T)$ instead of $\dist(\set{s},T)$. We write $B(x,r)$ and $\cl{B}(x,r)$ respectively for the open and closed balls with centre $x\in\R^{d}$ and radius $r\geq 0$. Moreover, we use the same notation for neighbourhoods of sets, i.e, $B(A,r):=\bigcup_{x\in A}B(x,r)$, where $A\subseteq\R^d$, and similarly for $\cl{B}(A,r)$.

\paragraph{Set related notions.} The cardinality of a set $A$ will be denoted by $\abs{A}$. For $m\in\N$ we let $[m]:=\set{1,2,\ldots,m}$. We also write $\R^{+}$ for the set of positive real numbers. 

\paragraph{Measures.} The symbol $\leb$ will be used to denote the Lebesgue measure on the given Euclidean space $\R^{d}$. Given a measurable function $\rho \colon Q\subseteq \R^{d}\to (0,\infty)$ we let $\rho\leb$ denote the measure on $Q$ defined by
\[
\rho\leb(E)=\niceint{E}{\rho}{\leb},\qquad E\subseteq Q.
\]
Moreover, if $f\colon Q\to\R^{d}$ is a mapping and $\mu$ is a measure on $Q$, we write  $f_{\sharp}\mu$ for the pushforward measure on $f(Q)$
\[
f_{\sharp}\mu(G):=\mu(f^{-1}(G)),\qquad G\subseteq f(Q).
\]

\paragraph{The displacement class of two separated nets.} We also introduce some notation to conveniently capture the $\phi$-displacement equivalences of two separated nets. 
\begin{define}\label{def:disp_cls}
	Let $X,Y\subseteq \R^{d}$ be separated nets. By $\disp_{R}(X,Y)$, we denote the class of increasing, concave functions $\phi\colon (0,\infty)\to (0,\infty)$ for which $X$ and $Y$ are $\phi$-displacement equivalent, according to Definition~\ref{def:disp_eq}.
\end{define}
\subsection*{Key properties of $\phi$-displacement equivalence.}
The next proposition records some sufficient conditions for deriving information on the growth of $\disp_{R}(f^{-1})$ from that of $\disp_{R}(f)$. 
\begin{prop}\label{prop:BD_of_func_and_its_inverse}
Let $X,Y$ be two separated nets in $\R^d$, $\phi\colon (0,\infty)\to (0,\infty)$ be an increasing concave function satisfying $\phi(R)\in o(R)$ 
and let $f\colon X \to Y$ be an injection with $\disp_{R}(f)\leq \phi(R)$ for every $R>0$. Then $\disp_{R}(f^{-1})\in O(\phi(R))$.
\end{prop}
\begin{proof}	
	The assumption $\disp_{R}(f)\leq \phi(R)$ implies that $\enorm{f(x)}\geq\enorm{x}-\phi(\enorm{x})$ for every $x\in X$ and by $\phi(R)\in o(R)$ there is $R_0>0$ such that for every $x\in X$ with $\enorm{x}\geq R_0$ it holds that $\enorm{x}-\phi(\enorm{x})\geq\enorm{x}/2$. Hence, using the concavity and the monotonicity of $\phi$, we can deduce that
	\[
	\enorm{x-f(x)}\leq\phi(\enorm{x})\leq 2\cdot \phi\br*{\frac{\enorm{x}}{2}}\leq 2\cdot\phi(\enorm{f(x)}),
	\]
	for every $x\in X$ with $\enorm{x}\geq R_{0}$, which proves that $\disp_R(f^{-1})\in O(\phi(R))$.\qedhere
\end{proof}
\begin{cor}\label{cor:sublinear_disp}
	Let $X,Y$ be two separated nets in $\R^{d}$ and $f\colon X\to Y$ be an injection with $\disp_{R}(f)\in o(R)$. Then $\disp_{R}(f^{-1})\in o(R)$.
\end{cor}
\begin{proof}
Let $\Psi\colon (0,\infty)\to (0,\infty)$ be a concave majorant of the function $R\mapsto \disp_{R}(f)$ with $\Psi(R)\in o(R)$ provided by Lemma~\ref{lemma:concave_maj}. We may now apply Proposition~\ref{prop:BD_of_func_and_its_inverse} to $\phi=\Psi$ and $f$ to verify the corollary.
\end{proof}
 The next example shows that if the assumption $\phi(R)\in o(R)$ in Proposition~\ref{prop:BD_of_func_and_its_inverse} is weakened to $\phi(R)\in O(R)$, then the proposition fails.
 It also shows, in contrast to Corollary~\ref{cor:sublinear_disp}, that no conclusion on the asymptotic class of $\disp_{R}(f^{-1})$ may be derived from the condition $\disp_{R}(f)\in O(R)$.
\begin{example}
	Let $\zeta\colon (0,\infty)\to (0,\infty)$ be an increasing function. Then there exist separated nets $X,Y\subseteq \R$ and a bijection $f\colon X\to Y$ such that $\disp_{R}(f)\in O(R)$ and $\disp_{R}(f^{-1})\notin O(\zeta(R))$. 
\end{example}
\begin{proof}
Let $X':=2\Z$ and $Y':=\Z$. Let $\psi\colon\N\to\frac{1}{2}+\N$ be any strictly increasing function and define $S_k:=\set{\psi(n)\colon n\in\N, n\geq k}$ for $k\in\N$. Finally, we set $X:=X'\cup S_2$ and $Y:=Y'\cup S_1$. Obviously, $X,Y$ are separated nets in $\R$. Now we can define a bijection $f\colon X\to Y$ as follows:
	\[
	f(x):=
	\begin{cases}
	\begin{array}{ll}
	\frac{1}{2}x &\text{ if }x\in X',\\
	\psi(n-1) \qquad\qquad &\text{ if }x=\psi(n).
	\end{array}	
	\end{cases}
	\]
	Clearly, $\disp_R(f)\in O(R)$, but $\disp_{\psi(n-1)}\br*{f^{-1}}\geq\psi(n)-\psi(n-1)$. It remains to restrict the choice of $\psi$ so that $\psi(n)-\psi(n-1)\geq n\zeta(\psi(n-1))$ for all $n\geq 2$.  
\end{proof}
To finish Section~\ref{sec:prelim}, we prove two results announced in the introduction; their statements are repeated here for the reader's convenience.
\linear
\begin{proof}
	We will assume that $\mb{0}\notin X,Y$; this can be ensured by an arbitrarily small shift. Then we observe that the condition $\disp_R(h)\in O(R)$ for a mapping $h\colon Z\to \R^{d}$ defined on a separated set $Z\subseteq \R^{d}\setminus\set{\mb{0}}$ is equivalent to the condition that there is $C>0$ such that
	\begin{equation}\label{eq:disp_lin_in_norm}
	\enorm{x-h(x)}\leq C\enorm{x}\qquad\forall x\in Z.
	\end{equation}
	Next we observe that the claim holds for $X$ and $Y$ if and only if there are $r_1, r_2>0$ such that it holds for $r_1X$ and $r_2Y$; assume that $g\colon r_1X\to r_2Y$ is a bijection and $C>0$ satisfies \eqref{eq:disp_lin_in_norm} for $g$. Then $f\colon X\to Y$ defined as $f(x):=\frac{1}{r_2}g(r_1x)$ is also a bijection and satisfies
	\begin{align*}
	\enorm{f(x)-x}&=\frac{1}{r_2}\enorm{g(r_1x)-r_2x}\leq\frac{\enorm{g(r_1x)-r_1x}+\enorm{r_1x-r_2x}}{r_2}\\
	&\leq\frac{Cr_1\enorm{x}+\abs{r_1-r_2}\enorm{x}}{r_2}=\br*{\frac{Cr_1+\abs{r_1-r_2}}{r_2}}\enorm{x}
	\end{align*}
	for every $x\in X$.
	
	Moreover, note that it is enough to prove that for every $X,Y$ there is always an \emph{injection} $f\colon X\to Y$ satisfying $\disp_R(f),\disp_R(f^{-1})\in O(R)$ instead of a \emph{bijection}---the result then follows by Rado's version of Hall's marriage theorem \cite{Rado-Hall_marriage} from infinite graph theory. Given two injections $f_X\colon X\to Y$ and $f_Y\colon Y\to X$ we can define a binary relation $E\subseteq X\times Y$ so that $\set{x,y}\in E$ if and only if $f_X(x)=y$ or $f_Y(y)=x$. Thus, $E$ is the union of the graphs of $f_X$ and $f_Y^{-1}$. By Rado's theorem there is a bijection $f\colon X\to Y$ such that $(\set{x,f(x)})_{x\in X}\subseteq E$ and then the condition $\disp_R(h)\in O(R)$ for every $h\in\set{f_X, f_X^{-1}, f_Y, f_Y^{-1}}$ ensures that $\disp_{R}(f),\disp_{R}(f^{-1})\in O(R)$.

	Now let $s>0$ stand for the separation of $X$ and $b>0$ for the net constant of $Y$.
	We choose $r>0$ such that $2rb<s$.
	For every $x\in X$ we find $g(x)\in rY$ such that $\enorm{x-g(x)}\leq rb$. As $2rb<s$, if $g(x)=g(x')$, then $x=x'$ for any $x,x'\in X$. Thus, $g$ is injective and the three observations above finish the proof.
\end{proof}

\equivrel
\begin{proof}
	Reflexivity is obvious. The symmetry of $\phi$-displacement equivalence follows from Proposition~\ref{prop:BD_of_func_and_its_inverse} if $\phi(R)\in o(R)$ and from Proposition~\ref{p:disp_always_linear} otherwise. To verify the transitivity, consider separated nets $X,Y,Z$ of $\R^{d}$ for which $X$ and $Y$ are $\phi$-displacement equivalent and $Y$ and $Z$ are $\phi$-displacement equivalent. Let the bijections $f\colon X\to Y$ and $g\colon Y\to Z$ witness this. Then $g\circ f$ is a bijection $X\to Z$ and there is a constant $K>0$ such that $\disp_{R}(f),\disp_{R}(g)\leq K\phi(R)$ for all $R>1$. 
	Let $R>1$ and $x\in X\cap \overline{B}(\mb{0},R)$. Then,
	\begin{multline*}
	\enorm{g\circ f(x)-x}\leq\enorm{g(f(x))-f(x)}+\enorm{f(x)-x}\\
	\leq \disp_{R+K\phi(R)}(g)+\disp_{R}(f)\leq 2K\phi(R+K\phi(R))\leq K'\phi(R),
	\end{multline*}

	for some constant $K'>0$ independent of $R$ and $x$. 
	The existence of $K'$ satisfying the last inequality is due to the conditions on $\phi$.
\end{proof}
	
	\section{Negative Results}\label{sec:neg}
The present section deals with obstructions to the existence of a bijection $f\colon X\to Y$ between two separated nets $X,Y$ in $\R^{d}$ with $\disp_{R}(f)\in o(R)$. The first lemma establishes that, in the case that $Y=\Z^{d}$ and such a bijection $f\colon X\to \Z^{d}$ exists, the separated net $X$ is forced to have quite a special property. In particular it is easy to come up with examples of $X$ not having the property described in the next lemma and thus not admitting any bijection $f\colon X\to\Z^{d}$ with $\disp_{R}(f)\in o(R)$.
\begin{lemma}\label{lem:disp_sublin_necessary}
	Let $X$ be a separated net in $\R^d$ and let $f\colon X\to \Z^d$ be a bijection such that $\disp_R(f)\in o(R)$.
	For any $r>0$ let 
	\[
	\mu_r(S):=\frac{1}{r^d}\abs{rS\cap X}, \qquad S\subseteq \cl{B}(\mb{0},1),
	\]
	stand for a normalised counting measure supported on the set $\frac{1}{r}X\cap \cl{B}(\mb{0},1)$ and let $(R_n)_{n\in\N}\subset\R^+$ be a sequence converging to infinity.
	Then the sequence $\br*{\mu_{R_n}}_{n\in\N}$ converges weakly to $\rest{\leb}{\cl{B}(\mb{0},1)}$.

\end{lemma}
\begin{proof}
	We write $\cl{B}:=\cl{B}(\mb{0},1)$.
	Let $s,b>0$ be the separation and the net constants of $X$, respectively. We set $X_n:=\frac{1}{R_n}X\cap \cl{B}$ and observe that each $X_n$ is an $\frac{s}{R_n}$-separated $\frac{2b}{R_n}$-net of $\cl{B}$.
	
	Next we define $f_n\colon X_n\to\R^d$ as $f_n(x):=\frac{1}{R_n}f(R_n x)$.
	Then the assumption $\disp_R(f)\in o(R)$ implies that
	\begin{equation}\label{eq:f_n_disp_vanishes}
	\lnorm{\infty}{f_n-\id}=\frac{1}{R_n}\lnorm{\infty}{f\circ R_n\id-R_n\id}\stackrel{n\to\infty}{\longrightarrow}0.
	\end{equation}
	In other words, $\lnorm{\infty}{f_n-\id}\in o(1)$. We also observe that $f_n$ is $o(R_n)$-Lipschitz: for any $x,y\in X_n$ it holds that
	\begin{align*}
	\enorm{f_n(x)-f_n(y)}&\leq\enorm{f_n(x)-x}+\enorm{f_n(y)-y}+\enorm{x-y}\\
	&\leq 2\lnorm{\infty}{f_n-\id}+\enorm{x-y}.
	\end{align*}
	Applying \eqref{eq:f_n_disp_vanishes}, we get that
	\[
	\frac{\enorm{f_n(x)-f_n(y)}}{\enorm{x-y}}\leq 1+\frac{o(1)}{\enorm{x-y}}.
	\]
	As $X_n$ is $\frac{s}{R_n}$-separated, the right-hand side above belongs to $o(R_n)$.
	
	Therefore, using Kirszbraun's Theorem~\cite{Kirszbraun1934}, each $f_n$ can be extended to an $o(R_n)$-Lipschitz mapping $\cl{f}_n\colon \cl{B}\to\R^d$.
	Now for any $x\in\cl{B}$ we choose $x_n\in X_n$ such that $\enorm{x-x_n}\leq\frac{2b}{R_n}$.
	Considering that $f_n(x_n)=\cl{f}_n(x_n)$ and \eqref{eq:f_n_disp_vanishes} we get that
	\begin{align*}
	\enorm{\cl{f}_n(x)-x}&\leq\enorm{\cl{f}_n(x)-\cl{f}_n(x_n)}+\enorm{f_n(x_n)-x_n}+\enorm{x_n-x}\\
	&\leq o(1)+\frac{2b}{R_n}\stackrel{n\to\infty}{\longrightarrow}0,
	\end{align*}
	where the $o(1)$ expression above is independent of $x$. This shows that $\cl{f}_{n}$ converges uniformly to $\id|_{\cl{B}}$.
	
	As a shortcut, we write $\mu_n:=\mu_{R_n}$. 
	By an application of Prokhorov's theorem, we observe that the sequence $(\mu_{n})$ converges weakly to the Lebesgue measure on $\cl{B}$ if and only if all of its weakly convergent subsequences do. Therefore, it is enough to verify the assertion of the lemma for an arbitrary weakly convergent subsequence of $(\mu_{n})$. We may assume, without loss of generality, that this given weakly convergent subsequence is the original sequence $(\mu_{n})$ and write $\mu$ for its weak limit.
	Using \cite[Lem.~5.6]{DKK2018} we get that $\br*{\cl{f}_n}_\sharp (\mu_n)$ converges weakly to $\br*{\rest{\id}{\cl{B}}}_\sharp (\mu)=\mu$. Consequently, it suffices to prove that $\br*{\cl{f}_n}_\sharp (\mu_n)$ converges weakly to $\rest{\leb}{\cl{B}}$.

	We define
	\[
	\varepsilon(R):=\sup_{R'\geq R}\frac{\disp_{R'}(f)}{R'}.
	\]
	The definition implies that $\varepsilon$ is decreasing and from the assumption $\disp_R(f)\in o(R)$ it follows that $\varepsilon(R_n)$ goes to zero as $n$ goes to infinity.
	For every $x\in X$ it holds that $\enorm{f(x)}\geq \enorm{x}-\varepsilon(\enorm{x})\enorm{x}$. This inequality in combination with the bijectivity of $f\colon X\to \Z^{d}$ and the fact that $\varepsilon$ is decreasing implies 	
	\begin{equation}\label{eq:increas_ball_cover}
	f\br*{X\cap\cl{B}(\mb{0},R)}\supseteq\Z^d\cap \cl{B}(\mb{0},(1-\varepsilon(R))R).
	\end{equation}
	for every $R>0$, where the ball on the right hand side should be interpreted as the empty set if its radius is negative. Indeed, observe that any point in the set on the right hand side has the form $f(x)$ for some $x\in X$ which satisfies $(1-\varepsilon(R))R>\enorm{f(x)}\geq \enorm{x}(1-\varepsilon(\enorm{x}))$ and therefore $\enorm{x}<R$. 
	
	Now we compare $\br*{\cl{f}_n}_\sharp(\mu_n)$ to the standard normalised counting measure $\nu_n$ supported on $\frac{1}{R_n}\Z^d$, i.e.,
	\[
	\nu_n(S):=\frac{1}{R_n^d}\abs{S\cap\frac{1}{R_n}\Z^d}\qquad\text{ for } S\subseteq\R^d.
	\]
	It is clear that $\nu_n\rightharpoonup\leb$. Thus, it suffices to verify that for any continuous function $\varphi\colon\R^d\to\R$ with compact support it holds that
	\[
	\abs{\niceint{\cl{f}_n\br*{\cl{B}}}{\varphi}{\br*{\cl{f}_n}_\sharp(\mu_n)}-\niceint{\cl{B}}{\varphi}{\nu_n}}\stackrel{n\to\infty}{\longrightarrow}0.
	\]
	Since $\br*{\cl{f}_n}_\sharp(\mu_n)$ is supported on $f_n(X_n)\subset\frac{1}{R_n}\Z^d$, we can rewrite the absolute value above as
	\[
	\frac{1}{R_n^d}\abs{\sum_{x\in f_n(X_n)\subset\frac{1}{R_n}\Z^d}\varphi(x)-\sum_{x\in\cl{B}\cap\frac{1}{R_n}\Z^d}\varphi(x)}.
	\]
	This expression can be bounded above by
	\begin{equation}\label{eq:weak_conv_upper_set_diff}
	\frac{1}{R_n^d}\lnorm{\infty}{\varphi}\abs{f_n(X_n)\Delta\br*{\cl{B}\cap\frac{1}{R_n}\Z^d}}.
	\end{equation}
	Further, we argue that \eqref{eq:weak_conv_upper_set_diff} can be bounded above by
	\begin{equation}\label{eq:weak_conv_set_diff_both_halves}
	\lnorm{\infty}{\varphi}\frac{\abs{\frac{1}{R_n}\Z^d\cap\cl{B}(\mb{0},1+\varepsilon(R_n))\setminus\cl{B}(\mb{0},1-\varepsilon(R_n))}}{R_n^d}.
	\end{equation}
	For every $n\in\N$
	\eqref{eq:increas_ball_cover} implies that
	$f_n\br*{X_n}\supseteq\frac{1}{R_n}\Z^d\cap\cl{B}\br*{\mb{0},1-\varepsilon(R_n)}$.
	Therefore,
	\begin{equation}\label{eq:weak_conv_set_diff_one_half}
	\br*{\cl{B}\cap\frac{1}{R_n}\Z^d}\setminus f_n(X_n)\subseteq\frac{1}{R_n}\Z^d\cap\cl{B}\setminus\cl{B}(\mb{0},1-\varepsilon(R_n)).
	\end{equation}
	
	Using the definition of $\varepsilon(R_n)$ we immediately get that for any $x\in X\cap\cl{B}(\mb{0},R_n)$ it holds that $\enorm{f(x)}\leq\enorm{x}+\varepsilon(R_n)R_n\leq(1+\varepsilon(R_n))R_n$.
	Therefore, $f(X\cap\cl{B}(\mb{0},R_n))\subseteq \Z^d\cap\cl{B}\br*{\mb{0},(1+\varepsilon(R_n))R_n}$. Consequently, we deduce that
	\[
	f_n(X_n)\subseteq\frac{1}{R_n}\Z^d\cap \cl{B}(\mb{0},1+\varepsilon(R_n)),
	\]
	which together with \eqref{eq:weak_conv_set_diff_one_half} proves \eqref{eq:weak_conv_set_diff_both_halves}.
	
	By centering an axes-aligned cube of side length $\frac{1}{R_n}$ at each point of the set$\frac{1}{R_n}\Z^d\cap\cl{B}(\mb{0},1+\varepsilon(R_n))\setminus\cl{B}(\mb{0},1-\varepsilon(R_n))$, we see that
	\[
	\abs{\frac{1}{R_n}\Z^d\cap\cl{B}(\mb{0},1+\varepsilon(R_n))\setminus\cl{B}(\mb{0},1-\varepsilon(R_n))}\leq R_n^d\leb\br*{\cl{B}\br*{\partial\cl{B},\varepsilon(R_n)+\frac{\sqrt{d}}{2R_n}}}.
	\]
	The last quantity is easily seen to be of order $R_n^d\cdot O\br*{\varepsilon(R_n)+\frac{1}{R_n}}$.
	This implies that the upper bound of \eqref{eq:weak_conv_set_diff_both_halves}, and thus, also \eqref{eq:weak_conv_upper_set_diff} go to zero as $n$ goes to infinity.
\end{proof}
We now recall the notion of natural density of a separated net. We will see that the natural density of separated nets is invariant under bijections $f$ with $\disp_{R}(f)\in o(R)$.
\begin{define}\label{def:natural_density}
	Let $X$ be a separated net in $\R^d$. Then its \emph{natural density}\footnote{Sometimes the term \emph{asymptotic density} is used instead in the literature.}, denoted by $\alpha(X)$, is defined as
	\[
	\alpha(X):=\lim_{R\to\infty}\frac{\abs{X\cap\cl{B}(\mb{0},R)}}{\leb\br*{\cl{B}(\mb{0},R)}},
	\]
	provided the limit exists; otherwise it is undefined. 
\end{define}
\begin{prop}\label{prop:nat_dens_different}
	Let $X,Y$ be two separated nets in $\R^d$ such that either $\alpha(X)\neq\alpha(Y)$, or exactly one of $\alpha(X), \alpha(Y)$ is not defined. Then there is no bijection $f\colon X\to Y$ with $\disp_R(f)\in o(R)$.
\end{prop}
\begin{proof}
	The assumption on $\alpha(X)$ and $\alpha(Y)$ implies that there is an unbounded, increasing sequence $(R_{n})_{n\in\N}$ such that 
		\[
		L:=\lim_{n\to\infty}\frac{\abs{X\cap\cl{B}(\mb{0},R_n)}}{\abs{Y\cap\cl{B}(\mb{0},R_n)}}
		\]
		is defined, but $L\neq 1$.
	We may assume without loss of generality that $L>1$. Otherwise just interchange $X$ and $Y$ and use Corollary~\ref{cor:sublinear_disp}. We choose $C\in (1,L)$ and find $n_0\in\N$ such that for every $n\geq n_0$ it holds that $\abs{X\cap\cl{B}(\mb{0},R_n)}\geq C \abs{Y\cap\cl{B}(\mb{0},R_n)}$. Because $Y$ is a separated net, there is $K>1$ and $n_1\in\N$ such that $\abs{Y\cap\cl{B}(\mb{0},KR_n)}< C\abs{Y\cap\cl{B}(\mb{0},R_n)}$ for every $n\geq n_1$.
	Therefore, for every $n\geq\max\set{n_0,n_1}$ we see that $\abs{X\cap\cl{B}(\mb{0},R_n)}>\abs{Y\cap\cl{B}(\mb{0},K R_n)}$, and thus, there must be $x_n\in X\cap\cl{B}(\mb{0},R_n)$ such that $\enorm{f(x_n)}>K R_n$. Consequently, $\enorm{x_n-f(x_n)}\geq\enorm{f(x_n)}-\enorm{x_n}\geq (K-1)R_n$.
\end{proof}

In view of Proposition~\ref{prop:nat_dens_different} it is natural to ask whether for two separated nets $X,Y\subseteq \R^{d}$ the condition that both natural densities $\alpha(X)$ and $\alpha(Y)$ are well defined and coincide is sufficient for the existence of a bijection $f\colon X\to Y$ with $\disp_{R}(f)\in o(R)$. We finish this section with an example which demonstrates that this is not the case:
\begin{example}
	There is a separated net $X$ in $\R^d$ such that $\alpha(X)=\alpha(\Z^d)$, but there is no bijection $f\colon X\to\Z^d$ with $\disp_R(f)\in o(R)$.
\end{example}
\begin{proof}
	Fix a hyperplane $H$ going through $\mb{0}$. We will denote the closed positive and the open negative half-spaces that it determines by $H^+$ and $H^-$, respectively. Moreover, fix $c\in (1,2)$ and define
	\[
	X:=(c^{-\frac{1}{d}}\Z^d\cap H^+)\cup((2-c)^{-\frac{1}{d}}\Z^d\cap H^-).
	\]
	Then, clearly, $\mu_{R_n}$ defined as in the statement of Lemma~\ref{lem:disp_sublin_necessary} converges weakly to the measure $c\rest{\leb}{\cl{B}\cap H^+}+(2-c)\rest{\leb}{\cl{B}\cap H^-}\neq\rest{\leb}{\cl{B}}$. On the other hand, $\alpha(X)=\alpha(\Z^d)$ by construction.
	Thus, Lemma~\ref{lem:disp_sublin_necessary} finishes the proof.
\end{proof}
	
	\section{The spectrum of \texorpdfstring{$\phi$}{\unichar{"03D5}}-displacement equivalence.}\label{sec:disp_class}
In the present section we prove Theorem~\ref{thm:displacement_class}:
\dispcls
Let us begin working towards a proof of Theorem~\ref{thm:displacement_class}. The proof is based on the following construction, which we present in a bit more general form than what is strictly needed for the proof of Theorem~\ref{thm:displacement_class}:
\begin{construction}\label{construction:BD_slow_growth_new}
Let $X$ be a separated net in $\R^d$ and let $(R_i)_{i\in\N}\subset \R^{+}$ be a strictly increasing sequence converging to infinity. Moreover, let $\phi\colon(0,\infty)\to(0,\infty)$ be an unbounded increasing function. The aim is to construct a set $Y$ in $\R^d$ which will, roughly speaking, be a piecewise rescaled version of $X$ and such that $\disp_{R}(Y,X)\subseteq\Omega(\phi(R))$. In the applications, we will choose $\phi$ and $(R_i)_{i\in\N}$ in a way that will ensure that $Y$ is a separated net. However, the construction described here is more general.

Formally, we will construct $Y$ as an image of $X$.
For any $R>0$ we set $\cl{R}:=R+\phi(R)$. We also define $\cl{R}_0:=R_0:=0$.
The desired mapping $g\colon X\to\R^d$ will be radial, so we first define its radial part $\gamma\colon [0,\infty)\to [0,\infty)$. We set $\gamma(\cl{R}_i):=R_i$ and prescribe that in between these specified values the function $\gamma$ interpolates linearly. Thus, $\gamma$ is a piecewise linear function with breaks precisely at the points $\cl{R}_i$.
Finally, we define $g(x):=\frac{\gamma(\enorm{x})}{\enorm{x}}x$ and $Y:=g(X)$.

For later use we introduce a sequence $(c_i)_{i\in\N}$ representing the slopes of $\gamma$. That is, for every $i\in\N$ we require that $\gamma(\cl{R}_i)=\gamma(\cl{R}_{i-1})+c_i(\cl{R}_i-\cl{R}_{i-1})$. This is equivalent to setting $c_i:=\frac{\gamma(\cl{R}_i)-\gamma(\cl{R}_{i-1})}{\cl{R}_i-\cl{R}_{i-1}}=\frac{R_i-R_{i-1}}{\cl{R}_i-\cl{R}_{i-1}}$.

We also record the maximum distance between consecutive `spherical layers' in $X$.
Let $\set{\ell_1<\ell_2<\ldots<\ell_k<\ldots}:=\set{\enorm{x}\colon x\in X}$. Additionally, we put $\ell_0:=0$. Then we define 
$s:=\sup\set{\ell_{k}-\ell_{k-1}\colon k\in\N}$.
Since $X$ is a net, $s$ is finite.
\end{construction}

\begin{prop}\label{prop:constr_gives_sep_net}
Assume, additionally to the assumptions of Construction~\ref{construction:BD_slow_growth_new}, that $\phi(R)\in O(R)$ and that there is $K>1$ such that $R_{i}\geq KR_{i-1}$ for every $i\in\N$.
Then $\gamma$ and $g$ are bilipschitz and $Y$ is a separated net. 
\end{prop}
\begin{proof}
Assuming that $\gamma$ is bilipschitz, it is easy to see that $g$ is bilipschitz as well, as $g$ is a radial map with radial part $\gamma$; just consider the points in spherical coordinates. Moreover, a bilipschitz image in $\R^{d}$ of a separated net in $\R^{d}$ is a separated net in $\R^{d}$.

The function $\gamma$ is bilipschitz if and only if the sequence $(c_i)_{i\in\N}$ is bounded and bounded away from zero.
As $\phi$ is increasing and $(R_i)_{i\in\N}$ is strictly increasing, we immediately obtain
\[
c_i=\frac{R_i-R_{i-1}}{\cl{R}_{i}-\cl{R}_{i-1}}=\frac{R_i-R_{i-1}}{R_{i}-R_{i-1}+\phi(R_i)-\phi(R_{i-1})}\leq 1.
\]

By the assumption on $\phi$ and the definition $\cl{R}=R+\phi(R)$ there is $C>1$ such that $R\leq\cl{R}\leq CR$ for every $R\geq R_1$.
We note that $(\cl{R}_i)_{i\in\N}$ is increasing. Using the assumption on the growth of $(R_{j})_{j\in\N}$, we obtain
\[
c_i=\frac{R_i-R_{i-1}}{\cl{R}_{i}-\cl{R}_{i-1}}\geq\frac{R_i-R_i/K}{\cl{R}_i}\geq\frac{K-1}{CK}>0.\qedhere
\]
\end{proof}

\begin{lemma}\label{lem:BD_constr_lower_bound}
Let $\phi, (R_i)_{i\in\N}, X,Y$ and $s$ be as in Construction~\ref{construction:BD_slow_growth_new} and let $f\colon Y\to X$ be an injective mapping. If, in addition, there is $K>0$ such that $\phi(R_{i+1})\leq K\phi(R_{i})$ for every $i\in\N$, then $\disp_R(f)\in \Omega(\phi(R))$. 
\end{lemma}
\begin{proof}
By Construction~\ref{construction:BD_slow_growth_new}, for every $i\in\N$ it holds that 
\[
\abs{X\cap \cl{B}(\mb{0},R_i+\phi(R_i))}=\abs{Y\cap\cl{B}\br*{\mb{0}, R_i}}.
\]
This implies that $\disp_{R_i}(f)\geq R_i+\phi(R_i)-s-R_i=\phi(R_i)-s$. 

Let $R>R_1$ be given and let $i\in\N$ be the unique index such that $R_{i-1}<R\leq R_{i}$. Then using the assumption on the growth of $\phi(R_{i+1})$, we can write
\[
\disp_R(f)\geq\disp_{R_{i-1}}(f)\geq\phi(R_{i-1})-s\geq\frac{\phi(R_i)}{K}-s\geq\frac{\phi(R)}{K}-s.
\]
The last quantity is greater than, say, $\phi(R)/2K$ for every $R$ large enough.
\end{proof}

\begin{lemma}\label{lem:BD_constr_upper_bound}
Let $\phi, (R_i)_{i\in\N}, X, Y$ and $g$ be as in Construction~\ref{construction:BD_slow_growth_new}. If, in addition, there is $K>0$ such that $\phi(R_{i+1})\leq K\phi(R_{i})$ for every $i\in\N$, then $\disp_R(g)\in O(\phi(R))$.
\end{lemma}
\begin{proof}
Since $R_i$ is strictly increasing and $\phi$ is increasing, we have $c_i\leq 1$ 
for every $i\in\N$. Because $\gamma$ is a piecewise affine function with slopes $c_{i}\leq1$ and $\gamma(0)=0$ the distance from $\gamma$ to the identity is an increasing function (with respect to $[0,R]$ with $R$ variable). This, in turn, means that the displacement of $g$ on the ball $\cl{B}(\mb{0},R)$ is realised on the points of $X$ closest to the boundary of the ball.
Now, we immediately get the bound
\[
\disp_{\cl{R}_i}(g)\leq \cl{R}_i-\gamma(\cl{R}_i)=R_i+\phi(R_i)-R_i=\phi(R_i).
\]

Fix $R>R_1+\phi(R_1)$ and choose the smallest $i\in\N$ such that $R\leq R_i+\phi(R_i)$. Then the growth condition on $\phi(R_{i})$ allows us to derive the bound
\begin{align*}
	\disp_R(g)&\leq\disp_{R_i+\phi(R_i)}(g)\leq \phi(R_i)\leq K\phi(R_{i-1})\\
	&\leq K\phi(R_{i-1}+\phi(R_{i-1}))\leq K\phi(R),
\end{align*}
where the last inequality is true thanks to the choice of $i$.
\end{proof}

Finally, we are ready to finish off the proof of Theorem~\ref{thm:displacement_class}:
\begin{proof}[Proof of Theorem~\ref{thm:displacement_class}]
We may assume that $\phi$ is unbounded, otherwise we may simply choose $Y$ as a non-zero but small perturbation of $X$. Such $Y$ is $BD$, and thus, also BL equivalent to $X$, while every bijection $X\to Y$ needs to displace the perturbed points by a non-zero distance.

We choose any $K>1$ and set $R_i:=K^i$ for ever $i\in\N$. This choice satisfies all the assumptions on $(R_i)$ in Proposition~\ref{prop:constr_gives_sep_net} and Lemmas~\ref{lem:BD_constr_lower_bound} and \ref{lem:BD_constr_upper_bound}, where the $\phi(R_{i+1})\leq K\phi(R_{i})$ assumption of the latter two statements is satisfied due to the concavity of $\phi$.
We apply Construction~\ref{construction:BD_slow_growth_new} using these objects and obtain a set $Y$ and a bijection $g\colon X\to Y$. Proposition~\ref{prop:constr_gives_sep_net} says that $Y$ is a separated net and $g\colon X\to Y$ witnesses the BL equivalence of $X$ and $Y$.
Applying Lemma~\ref{lem:BD_constr_upper_bound} we get that $\disp_R(g)\in O(\phi(R))$, from which $\disp_{R}(g^{-1})\in O(\phi(R))$ follows via Proposition~\ref{prop:BD_of_func_and_its_inverse}. Now let $f\colon X\to Y$ be a bijection. By Lemma~\ref{lem:BD_constr_lower_bound}, it holds that $\disp_R(f^{-1})\in\Omega(\phi(R))$.
Let $\Psi\colon (0,\infty)\to(0,\infty)$ be a concave majorant of $t\mapsto \disp_{t}(f)$ with $\disp_{R}(f)\notin o(\Psi(R))$, given by Lemma~\ref{lemma:concave_maj}. Then, applying Proposition~\ref{prop:BD_of_func_and_its_inverse}, we infer that $\disp_{R}(f^{-1})\in O(\Psi(R))\cap \Omega(\phi(R))$, which implies $\phi(R)\in O(\Psi(R))$. This, together with $\disp_{R}(f)\notin o(\Psi(R))$, implies $\disp_{R}(f)\notin o(\phi(R))$.
\end{proof}
	
	\section{Continuously many, pairwise distinct BD equivalence classes.}\label{sec:BD_class}
The objective of the present section is to prove Theorem~\ref{thm:BD_equiv_classes}, whose statement we repeat for the reader's convenience:
\BDclasses
The proof of Theorem~\ref{thm:BD_equiv_classes} is based on the following proposition:
\begin{prop}\label{prop:distinct_disp_distinct_BD}
	Let $d\in\N$, $X$ be a separated net in $\R^{d}$ and $\phi_{1},\phi_{2}\colon (0,\infty)\to (0,\infty)$ be increasing, unbounded and concave functions such that $\phi_{i}(R)\in o(R)$ for $i\in[2]$ and $\phi_1(R)\in o(\phi_2(R))$. Let $Y_{1},Y_{2}\subseteq \R^{d}$ be separated nets 
	such that $\phi_1\in \disp_{R}(X,Y_{1})$	
	and $\disp_{R}(X,Y_{2})\cap o(\phi_2(R))=\emptyset$. Then $Y_{1}$ and $Y_{2}$ are BD non-equivalent.
\end{prop}
\begin{proof}
	Assume for a contradiction that $Y_{1}$ and $Y_{2}$ are BD equivalent and consider a bijection $f\colon Y_{2}\to Y_{1}$ for which
	\[
		\disp(f)=\sup_{x\in Y_{2}}\lnorm{2}{f(x)-x}<\infty.
	\]
Let $g\colon Y_{1}\to X$ be a bijection for which $\disp_{R}(g)\in O(\phi_{1})$
and let $K>0$ be sufficiently large so that $\disp_{R}(g)\leq K\phi_{1}(R)$ for all $R>1$. Then, we may define a bijection $h\colon Y_{2}\to X$ by $h:=g\circ f$. Let us estimate the asymptotic growth of $\disp_{R}(h)$: fix $R>1$ and $x\in Y_{2}\cap \overline{B}(0,R)$. Then $f(x)\in Y_{1}\cap \overline{B}(0,R+\disp(f))$, from which it follows that $\lnorm{2}{g(f(x))-f(x)}\leq K\phi_{1}(R+\disp(f))$. Now we may write
\begin{align*}
	\enorm{h(x)-x}&\leq \enorm{g(f(x))-f(x)}+\enorm{f(x)-x}\\
	&\leq K\phi_{1}(R+\disp(f))+\disp(f)\\
	&\leq K'\phi_{1}(R),
\end{align*}
which is true for some $K'>K$ independent of $R$.
We deduce that $h\colon Y_{2}\to X$ is a bijection satisfying $\disp_{R}(h)\in O(\phi_{1}(R))\subseteq o(\phi_{2}(R))$, contrary to 
$\disp_{R}(X,Y_{2})\cap o(\phi_2(R))=\emptyset$. 
\end{proof}

\begin{proof}[Proof of Theorem~\ref{thm:BD_equiv_classes}]
Fix $d\in\N$ and a representative $X$ of a given BL equivalence class of separated nets in $\R^{d}$. Let $\Lambda$ denote the set of all increasing, unbounded and concave functions $(0,\infty)\to (0,\infty)$. For each $\phi\in\Lambda$, we apply Theorem~\ref{thm:displacement_class} to obtain a separated net $Y_{\phi}$ in $\R^{d}$ belonging to the same BL equivalence class as $X$ and satisfying 
$\disp_{R}(X,Y)\cap o(\phi(R))=\emptyset$. 
Now, Proposition~\ref{prop:distinct_disp_distinct_BD} verifies that the family of separated nets $(Y_{\phi})_{\phi\in \Lambda}$ contains uncountably many pairwise BD non-equivalent separated nets.
\end{proof}
	
	\section{The intersection between the BL classes and the classes of bounded growth of displacement}\label{sec:BL_inter_disp}

In this section we prove Theorem~\ref{thm:bil_neq_slow_disp}:
\slowdisp
The proof of Theorem~\ref{thm:bil_neq_slow_disp} will require several lemmas, some of which are quite technical. Therefore, we first describe the main ideas of the proof informally.

The proof consists of three main ingredients. The first one is an observation that in order to construct a separated net $X$ in $\R^d$ that is not bilipschitz equivalent to $\Z^d$ it is possible to start with $\Z^d$ and modify it only inside a collection of pairwise disjoint cubes $(S_k)_{k\in\N}$ of increasing size; see Figure~\ref{f:BL_vs_BD}. The actual position of these cubes, as long as they remain disjoint, is irrelevant with respect to bilipschitz non-equivalence with $\Z^d$.

Moreover, we can actually get that there is even no bilipschitz injection $X\to\Z^d$ which would also be a bijection between a neighbourhood of $S_k$ and a certain neighbourhood of the image of $S_k$, for each $k$ separately; this is based on the work of Burago and Kleiner~\cite[Lemma~2.1]{BK1} and formalised in Lemma~\ref{lemma:BK} below.
\begin{figure}
	\begin{center}
		\includegraphics[width=\linewidth]{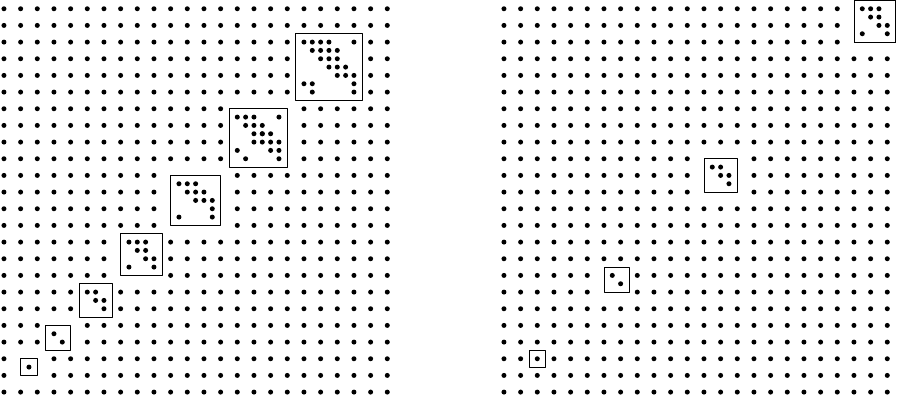}
		\caption{Examples of `very different' separated nets $X_1$ (left) and $X_2$ (right). The cubes in $(S_k^1)$ are the same as in $(S_k^2)$, but the distance of $S_k^1$ to the origin grows with $k$ much slower than that of $S_k^2$.}
		\label{f:BL_vs_BD}
	\end{center}
\end{figure}

The second ingredient is exploiting the last property of the construction of $X$ mentioned above. It sometimes allows us to rule out bilipschitz bijections between two different nets $X_{1}$ and $X_{2}$ arising in the way described above, instead of just bilipschitz bijections between $X$ and $\Z^d$. We will see that if $f\colon X_1\to X_2$ is a bijection such that infinitely many of the cubes $(S_k^1)$ used to construct $X_1$ are mapped by $f$ to parts of $X_2$ equal to $\Z^d$, then $f$ is not bilipschitz; see Lemma~\ref{lemma:bilip_nonequiv}. Thus, if we place the cubes $(S_k^1)$ `very differently' inside $\R^d$ in comparison to the cubes $(S_k^2)$ used to define $X_2$ (see Figure~\ref{f:BL_vs_BD}), we may hope that for every bilipschitz mapping $f\colon X_1\to \R^d$ there will be infinitely many $i\in\N$ such that the image $f(S_i^1)$ will miss all cubes in $(S_k^2)$; this is substantiated in Lemma~\ref{lemma:miss_inft_many}. Since it is possible to come up with uncountably many `very different' ways how to place the cubes $(S_k)$ in $\R^d$, we will obtain an uncountable family $(X_\psi)_{\psi\in\Lambda}$ of pairwise bilipschitz non-equivalent separated nets.

The last ingredient is responsible for showing that each of the nets $X_{\psi}$ in the family described above is $\phi$-displacement equivalent to $\Z^{d}$. We observe that the construction of each $X:=X_{\psi}$ inside its corresponding collection $(S_k)$ can ensure that $\abs{X\cap S_k}=\abs{\Z^d\cap S_k}$ for every $k\in\N$---this is the purpose of Lemma~\ref{lemma:Gammak}. Since outside $\bigcup S_{k}$ each $X$ is equal to $\Z^d$, this allows us to define a bijection $X\to\Z^d$ with controlled growth of displacement: Outside $(S_k)$ we use the identity function and inside $S_k$ we can use any bijection $X\cap S_k\to\Z^d\cap S_k$. The displacement of the resulting bijection on $\cl{B}(0,R)$ is then no larger than the diameter of the largest $S_k$ intersecting $\cl{B}(0,R)$; this is formalized in Lemma~\ref{lemma:disp_X_psi}.

We continue providing formal arguments for the claims outlined above.
\begin{lemma}\label{lemma:miss_inft_many}
Let $F\colon D\subseteq \R^{d}\to \R^{d}$ be a bilipschitz mapping, $\psi_1, \psi_2\colon(0,\infty)\to(0,\infty)$ be two increasing functions such that $\psi_2(R+K)\in o(\psi_1(R))$ for any fixed $K\in\N$ and $(U_k)_{k\in\N}$ be a sequence of
cubes in $\R^d$ with $\diam U_k$ increasing and $\diam U_k\in o(\psi_2(k))$.
Moreover, we assume that $g_1,g_2\colon \bigsqcup_{k\in\N}U_k\to\R^d$ are mappings such that
\begin{enumerate}[1.]
	\item\label{l_i:large_dist} $\dist\br*{g_1(U_k),g_1\br*{\bigcup_{j\neq k} U_{j}}}\geq \psi_1(k)$ for every $k\in\N$, 
	\item\label{l_i:small_dist} $\dist\br*{g_2(U_k),g_2\br*{\bigcup_{j\neq k} U_{j}}}=\dist\br*{g_2(U_k),g_2\br*{U_{k-1}}}= \psi_2(k)$ for every $k\geq 2$, 
	\item\label{l_i:piecewise_trans} $\rest{g_i}{U_k}$ is a translation for every $i=1,2$ and $k\in\N$.
\end{enumerate}
Then there are infinitely many $i\in\N$
such that
\[
F\br*{D\cap \bigcup_{k\in\N} g_1(U_k)}\cap g_2(U_i)=\emptyset.
\]
\end{lemma}
\begin{proof}
Since $F$ is defined only on the set $D$, in every application of $F$ in this proof the argument of $F$ should always be intersected with $D$ to ensure that the whole expression is well-defined; however, to improve the readability of formulas, we omit it.

We define $i(k):=\max\set{i\in\N\colon F(g_1(U_k))\cap g_2(U_i)\neq\emptyset}$; if the set over which the maximum is taken is empty, we set $i(k)$ to $\infty$. Let
\[C:=\max\set{\lip(F),\lip(F^{-1})}.\] 

We split the proof into two cases. First, we assume that there is $A\in\N$ such that for every $k\in\N$ there is $n:=n(k)\in\N, n\geq k$ such that $i(n)\leq n+A$. Fix $k\in\N$ and $n=n(k)$.
Condition \ref{l_i:small_dist} on $g_2$ implies that $\dist(g_2(U_{i(n)}), g_2(U_{i(n)+1}))\leq \psi_2(n+A+1)$.
From Condition~\ref{l_i:large_dist}
we get that $\dist\br*{F\circ g_1(U_{n}), F\circ g_1\br*{\bigcup_{j\neq n}U_{j}}}\geq\psi_1(n)/C$.
Next, we write
\begin{multline*}
\dist\br*{F\br*{\bigcup_{j\neq n}g_1(U_{j})}, g_2(U_{i(n)+1})}\\
\geq \frac{\psi_1(n)}{C}-\diam g_2(U_{i(n)})-\psi_2(n+A+1)-\diam g_2(U_{i(n)+1}).
\end{multline*}
Note that $\diam g_2(U_{i(n)}),\diam g_2(U_{i(n)+1}) \in o(\psi_2(n+A+1))$ according to the assumptions. Thanks to the assumption $\psi_2(R+A+1)\in o(\psi_1(R))$, we get that $F\br*{\bigcup_{j\in\N}g_1(U_{j})}\cap g_2(U_{i(n)+1})=\emptyset$ provided $k$ (and thus $n$) is large enough. This establishes the assertion in the present case.

Next we assume that for every $A\in\N$ it holds that $i(k)>k+A$ for every $k$ large enough. In particular, there exists $k_{0}\in\N$ such that $i(k)>k$ for every $k\geq k_0$. Moreover, we assume that $k_0$ is large enough so that whenever $k\geq k_0$ and $F\circ g_1(U_k)\cap g_2(U_{i'})\neq\emptyset$, we have $i'=i(k)$. This is possible, as
either no such indices $i'$ exist, or
$F\circ g_1(U_k)\cap g_2(U_{i(k)})\neq\emptyset$ and $\dist\br*{g_2(U_{i(k)}),g_2\br*{\bigcup_{j\neq i(k)} U_j}}=\psi_2(i(k))\geq \psi_2(k)$ according to Condition~\ref{l_i:small_dist} in the present case. Now it suffices to use the fact that $\diam F\circ g_1(U_k)\in o(\psi_2(k))$, which follows from the assumptions.

We continue by contradiction: assume that there is $i_0\in\N$ such that for every $i\geq i_0$ it holds that $F\br*{\bigcup_{k\in\N}g_1(U_k)}\cap g_2(U_i)\neq\emptyset$. We will also assume that $i_0>\max\set{i(j)\colon j\in\N, j\leq k_0, i(j)<\infty}$. Given the property of $i(\cdot)$ proven above this means that for every $i\geq i_0$ there is $k\geq k_0$ such that $i=i(k)$. Let $k_1\in\N$, $k_1\geq k_0$ be a number satisfying $i(k)\geq i_0$ for every $k\geq k_1$. Next we choose $K\in\N$ such that either $i(k)\leq k+K$, or $i(k)=\infty$
for every $k<k_1$. Furthermore, we choose $k_2\in\N, k_2\geq k_1$ large enough so that $i(k)>k+K$ for every $k\geq k_2$.
In consequence, for every $k\geq k_2$ the set
\[
\set{i\in\N\colon F\circ g_1\br*{\bigcup_{j\leq k}U_j}\cap g_2(U_i)\neq\emptyset}
\]
can contain at most $k-k_1$ numbers within the set $\set{k_1+K, \ldots, k+K}$. But this, in turn, means that there is $l\in\N, l>k$ such that $i(l)\leq k+K$. At the same time, $i(l)>l+K\geq k+K+1$; a contradiction.
\end{proof}

\begin{lemma}\label{lemma:BK}
	Let $\rho\colon [0,1]^{d}\to (0,\infty)$ be a measurable function with $0<\inf\rho\leq\sup\rho<\infty$ and the property that the equation $\Phi_{\sharp}\rho\leb=\leb|_{\Phi([0,1]^{d})}$ has no bilipschitz solutions $\Phi\colon [0,1]^{d}\to\R^{d}$. Let $(R_{k})_{k\in\N}$ and $(S_{k})_{k\in\N}$ be  sequences of pairwise disjoint cubes in $\R^{d}$ such that $\diam R_{k}$ and $\diam S_{k}$ are unbounded and increasing and $2S_{k}\subseteq R_{k}$ for every $k\in\N$, where $2S_{k}$ denotes the cube with the same midpoint as $S_{k}$ and sidelength twice the sidelength of $S_{k}$. For each $k\in\N$, let $\phi_{k}\colon \R^{d}\to\R^{d}$ denote the unique affine mapping $\R^{d}\to\R^{d}$ with scalar linear part satisfying $\phi_{k}([0,1]^{d})=S_{k}$. For each $k\in\N$, let $\Upsilon_{k}$ be a finite subset of $R_{k}$ such that $\bigcup_{k\in\N}\Upsilon_{k}$ is a separated net of $\bigcup_{k\in\N}R_{k}$ and the normalised counting measure on the set $\phi_{k}^{-1}(\Upsilon_{k}\cap S_{k})$ converges weakly to $\rho\leb$. Let $h\colon \bigcup_{k\in\N}\Upsilon_{k}\to \Z^{d}$ be an injective mapping such that 
	\begin{equation}\label{eq:buffer}
	B(h(\Upsilon_{k}\cap S_{k}),\diam S_{k})\cap \Z^{d}\subseteq h(\Upsilon_{k})
	\end{equation}
	for each $k\in\N$. Then $h$ is not bilipschitz. In fact, 
	\begin{equation}\label{eq:sup_bil_cst}
	\sup_{k\in\N}\max\set{\lip(h|_{\Upsilon_{k}}),\lip((h|_{\Upsilon_{k}})^{-1})}=\infty.
	\end{equation}
\end{lemma}
\begin{proof}
	The argument of the present proof in its original form is due to Burago and Kleiner; see \cite[Proof of Lemma~2.1]{BK1}. Moreover, a more detailed presentation of the argument is given by the present authors in \cite[Proof of Lemma~3.4]{dymond_kaluza2019highly}. Therefore, we present the first part of the proof here quite succinctly, leaving several verifications to the reader, which may be thought of as exercises. For further details, we refer the reader  to the works \cite{BK1} and \cite{dymond_kaluza2019highly}.
	
	Observe that
	\[
	\phi_{k}^{-1}(2S_{k})=\left[-\frac{1}{2},\frac{3}{2}\right]^{d}\supset [0,1]^{d}=\phi_{k}^{-1}(S_{k})
	\]
	for all $k$. Suppose for a contradiction that the supremum of \eqref{eq:sup_bil_cst} is finite. Then, denoting by $l_{k}$ the sidelength of the square $S_{k}$, we deduce that the mappings $f_{k}:=\frac{1}{l_{k}}h\circ \phi_{k}$, extended using Kirszbraun's theorem from 
	\[
	\cl{\Gamma}_{k}:=\phi_{k}^{-1}(\Upsilon_{k})\cap \left[-\frac{1}{2},\frac{3}{2}\right]^{d}
	\]
	 to the cube $\left[-\frac{1}{2},\frac{3}{2}\right]^{d}$, are uniformly Lipschitz and, after composing each $f_{k}$ with a translation if necessary so that the image of every $f_{k}$ contains $\mb{0}$, they are also uniformly bounded. Applying the Arzel\`a-Ascoli theorem, we may pass to a subsequence of $(f_{k})_{k\in\N}$ which converges uniformly to a Lipschitz mapping $f\colon \left[-\frac{1}{2},\frac{3}{2}\right]^{d}\to \R^{d}$. Using the fact that each $f_{k}$ is bilipschitz on the finer and finer net $\cl{\Gamma}_{k}$ of $\left[-\frac{1}{2},\frac{3}{2}\right]^{d}$, we deduce that $f$ is also bilipschitz. 
	
	Let $\mu_{k}$ denote the normalised counting measure on 
	\[
	\Gamma_{k}:=\phi_{k}^{-1}(\Upsilon_{k}\cap S_{k})
	\]
	so, by hypothesis, $\mu_{k}$ converges weakly to $\rho\leb$. We claim that the pushforward measures $(f_{k}|_{[0,1]^{d}})_{\sharp}\mu_{k}$ converge weakly to the Lebesgue measure on $f([0,1]^{d})$. This claim, together with the uniform convergence of $f_{k}$ to $f$, implies that $f_{\sharp}\rho\leb=\leb|_{f([0,1]^{d})}$, contrary to the hypothesis on $\rho$.
	
	Therefore, to complete the proof, it only remains to verify the claim, that is, to prove that $(f_{k}|_{[0,1]^{d}})_{\sharp}\mu_{k}$ converges weakly to $\leb|_{f([0,1]^{d})}$. This remaining part of the proof is more subtle. The argument we give here is not present in \cite{BK1}, but is an adaptation of \cite[Proof of Lemma~3.2]{dymond_kaluza2019highly}. Although the adaptation is quite simple, it requires good familiarity with the proof in \cite{dymond_kaluza2019highly} to construct it. Therefore, we provide more details here.
	
	Consider the sequence of measures
	\[
	\nu_{k}(A):=\frac{1}{l_{k}^{d}}\abs{A\cap \frac{1}{l_{k}}\Z^{d}}, \qquad A\subseteq \R^{d},\,k\in\N,
	\]
	which clearly converges weakly to the Lebesgue measure on $\R^{d}$. For a given continuous function $\varphi\colon\R^{d}\to\R$ with compact support we need to verify
		\begin{equation}\label{eq:weakconv}
		\abs{\niceint{f([0,1]^{d})}{\varphi}{\nu_{k}}-\niceint{{f}_k([0,1]^{d})}{\varphi}{(\rest{{f}_k}{[0,1]^{d}})_\sharp \mu_{k}}}\stackrel[k\to\infty]{}{\longrightarrow}0.
		\end{equation}

	We bound the expression in \eqref{eq:weakconv} above by the sum of two terms:
		\begin{align}\label{eq:weakconv_split}
		\begin{split}
		\abs{\niceint{f([0,1]^{d})}{\varphi}{\nu_{k}}-\niceint{{f}_k([0,1]^{d})}{\varphi}{\nu_{k}}}+\abs{\niceint{{f}_k([0,1]^{d})}{\varphi}{\nu_{k}}-\niceint{{f}_k([0,1]^{d})}{\varphi}{(\rest{{f}_k}{[0,1]^{d}})_\sharp \mu_{k}}}
		\end{split}
		\end{align}
	The first term is at most $\lnorm{\infty}{\varphi}\nu_{k}(f([0,1]^{d})\Delta {f}_{k}([0,1]^{d}))$, which vanishes as $k\to\infty$ due to the weak convergence of $\nu_{k}$ to $\leb$, the uniform convergence of $f_{k}$ to $f$ and the fact that $f$ is bilipschitz. We do not provide further details here; the verification is left as an exercise with reference to \cite[Lemma 3.1]{dymond_kaluza2019highly}. The second term may be bounded above by
		\begin{equation}\label{eq:second_term_bound}
		\frac{\lnorm{\infty}{\varphi}}{l_{k}^{d}}\abs{A_{k}},\qquad \text{where $A_{k}:={f}_{k}([0,1]^{d})\cap \frac{1}{l_{k}}\Z^{d}\setminus f_{k}(\Gamma_{k})$.}
		\end{equation}
	We will argue that
		\begin{equation}\label{eq:Ak_inclusion}
		A_{k}\subseteq \cl{B}\left(\partial f([0,1]^{d}),\lnorm{\infty}{{f}_{k}-f}\right)
		\end{equation}
	for all $k$ sufficiently large. Once this is established the quantity of \eqref{eq:second_term_bound} is seen to be at most
		\begin{equation*}
		\lnorm{\infty}{\varphi}\leb\br*{\cl{B}\br*{\partial f([0,1]^{d}),\lnorm{\infty}{{f}_{k}-f}+\frac{\sqrt{d}}{l_{k}}}},
		\end{equation*}
	which converges to zero as $k\to\infty$. Hence, to complete the verification of the weak convergence of $(\rest{{f}_{k}}{[0,1]^d})_{\sharp}\mu_{k}$ to $\leb|_{f([0,1]^{d})}$, we prove \eqref{eq:Ak_inclusion}.
	
	From now on we treat $k$ as fixed but sufficiently large. Recall that the sequence of mappings $f_{i}|_{\cl{\Gamma}_{i}}\colon \cl{\Gamma}_{i}\to\frac{1}{l_{i}}\Z^{d}$, $i\in\N$, is uniformly bilipschitz and set
	\begin{equation*}
	U:=\sup_{i\in\N}\max\set{\lip(f_{i}|_{\cl{\Gamma}_{i}}),\lip({f_{i}|_{\cl{\Gamma}_{i}}}^{-1})}<\infty.
	\end{equation*}
	Since the mappings $f_{i}\colon \left[-\frac{1}{2},\frac{3}{2}\right]^{d}\to\R^{d}$ were obtained as Kirszbraun's extensions of $f_{i}|_{\cl{\Gamma}_{i}}$, we additionally note that $\lip(f_{i})\leq U$ for all $i\in\N$.
	We also write $b$ for the maximum of the net constants of $\bigcup_{i=1}^{\infty}\Upsilon_{i}\cap S_{i}$ in $\bigcup_{i=1}^{\infty}S_{i}$ and of $\bigcup_{i=1}^{\infty}\Upsilon_{i}$ in $\bigcup_{i=1}^{\infty}R_{i}$. The condition \eqref{eq:buffer} 
	translates, after application of the homeomorphism $x\mapsto \frac{x}{l_{k}}$, to 
	\[
	B\left(f_{k}(\Gamma_{k}),\sqrt{d}\right)\cap \frac{1}{l_{k}}\Z^{d}\subseteq f_{k}(\phi_{k}^{-1}(\Upsilon_{k})).
	\]
	At the same time, $\Gamma_{k}$ is a $\frac{b}{l_{k}}$-net of $[0,1]^{d}$, so that $f_{k}([0,1]^{d})\subseteq \cl{B}(f_{k}(\Gamma_{k}),\frac{Ub}{l_{k}})$. Since $k$ is sufficiently large, it follows that 
	\[
	A_{k}\subseteq \left(B\left(f_{k}(\Gamma_{k}),\sqrt{d}\right)\cap \frac{1}{l_{k}}\Z^{d}\right)\setminus f_{k}(\Gamma_{k})\subseteq f_{k}(\phi_{k}^{-1}(\Upsilon_{k}))\setminus f_{k}(\Gamma_{k}).
	\]
	Thus, any point in $A_{k}$ has the form $f_{k}(x)$ for some 
	\[
	x\in \phi^{-1}_{k}(\Upsilon_{k}\setminus S_{k})=\cl{\Gamma}_{k}\setminus [0,1]^{d}.
	\]
	 If $f_{k}(x)\notin f([0,1]^{d})$ then $f_{k}(x)\in A_k\setminus f([0,1]^{d})\subseteq {f}_{k}([0,1]^{d})\setminus f([0,1]^{d})$, and therefore, $\dist(f_{k}(x),\partial f([0,1]^{d}))\leq\lnorm{\infty}{{f}_{k}-f}$.
	
	In the remaining case we have $f_{k}(x)\in f([0,1]^{d})$. Since $f$ is defined at $x\in \cl{\Gamma}_{k}\setminus [0,1]^{d}\subseteq \left[-\frac{1}{2},\frac{3}{2}\right]^{d}\setminus[0,1]^{d}$ and $f$ is injective, we additionally have $f(x)\notin f([0,1]^{d})$. Thus, we deduce that
			\[
			\dist(f_{k}(x),\partial f([0,1]^{d}))\leq\lnorm{2}{f_{k}(x)-f(x)}\leq \lnorm{\infty}{{f}_{k}-f},
			\]
		as required.
\end{proof}

\begin{lemma}\label{lemma:Gammak}
	Let $\rho\colon [0,1]^{d}\to (0,\infty)$ be a measurable function with $0<\inf\rho\leq\sup\rho<\infty$ and $\niceint{[0,1]^{d}}{\rho}{\leb}=1$. Let $(S_{k})_{k\in\N}$ be a sequence of pairwise disjoint cubes in $\R^{d}$ such that the sidelength $l_{k}\in\N$ of $S_{k}$ is unbounded and increasing. Let $(\phi_{k})_{k\in\N}$ denote the sequence of affine mappings $\phi_{k}$ with scalar linear part $l_{k}$ and $\phi_{k}([0,1]^{d})=S_{k}$. Then there exists a sequence $(\Xi_{k})_{k\in\N}$ of finite sets $\Xi_{k}\subseteq S_{k}$ with the following properties:
	\begin{enumerate}[(i)]
		\item\label{prop:crd} $\abs{\Xi_{k}}=l_{k}^{d}$ for every $k\in\N$,
		\item\label{prop:sep_net} $\bigcup_{k\in\N}\Xi_{k}$ is a separated net of $\bigcup_{k\in\N}S_{k}$,
		\item\label{prop:wk_cnv} The sequence $(\mu_{k})_{k\in\N}$, where $\mu_{k}$ is the normalised counting measure on the set $\phi_{k}^{-1}(\Xi_{k})$, converges weakly to $\rho\leb$.
	\end{enumerate}
\end{lemma}
\begin{proof}
	If property \eqref{prop:crd} is omitted, the proof is contained in \cite[Proof of Lemma~2.1]{BK1}; similar constructions are also given in \cite{DKK2018} and \cite{dymond_kaluza2019highly}. Getting property \eqref{prop:crd} only requires taking a little extra care in the construction of \cite[Proof of Lemma~2.1]{BK1}. Therefore, we present only minimal details here; the calculations and the verification of \eqref{prop:crd}--\eqref{prop:wk_cnv} are left to the reader.
	
	Let $m_{k}:=\floor*{\sqrt{l_{k}}}$ for $k\in\N$. Fix $k\in\N$. We describe how to obtain the set $\Xi_{k}\subseteq S_{k}$. Consider the standard partition $(T_{k,i})_{i\in[m_{k}^{d}]}$ of the cube $S_{k}$ into $m_{k}^{d}$ subcubes of equal size and choose a sequence $(n_{k,i})_{i\in[m_{k}^{d}]}$ satisfying
	\begin{align*}
		&n_{k,i}\in\set{\floor*{l_{k}^{d}\niceint{\phi_{k}^{-1}(T_{k,i})}{\rho}{\leb}},\floor*{l_{k}^{d}\niceint{\phi_{k}^{-1}(T_{k,i})}{\rho}{\leb}}+1},\,i\in[m_{k}^{d}],\\
		&\sum_{i\in [m_{k}^{d}]}n_{k,i}=l_{k}^{d}.
	\end{align*}
	It is now enough to define $\Xi_{k}$ so that $\abs{\Xi_{k}\cap T_{k,i}}=n_{k,i}$ for all $i\in[m_{k}^{d}]$ and the separation and net constants of $\Xi_{k}$ in $S_{k}$ may be bounded respectively below and above independently of $k$. For each $i\in[m_{k}^{d}]$, we suggest the following prescription of the set $\Xi_{k}\cap T_{k,i}$: imagine we have a pot containing $n_{k,i}$ points. In the first step, we take one point out of the pot and place it at the centre of the cube $T_{k,i}$. Assume now that $j\geq 1$ and that after $j$ steps we have placed exactly one point from the pot at the centre of each cube in each of the the first $j-1$ dyadic partitions of the cube $T_{k,i}$. In step $j+1$, we consider the $j$th dyadic partition of $T_{k,i}$ and arbitrarily transfer remaining points from the pot onto the vacant centres of each of the $2^{dj}$ cubes in this partition until either the pot is empty or all of the $2^{dj}$ centres are occupied. When the pot is empty, the procedure terminates and the placement of the $n_{k,i}$ points determines the set $\Xi_{k}\cap T_{k,i}$. 
\end{proof}

\begin{construction}\label{const_X_psi}
	Let $\rho\colon [0,1]^{d}\to (0,\infty)$ be a measurable function with $0<\inf\rho\leq\sup\rho<\infty$ and $\niceint{[0,1]^{d}}{\rho}{\leb}=1$, $\mathfrak{l}=(l_{k})_{k\in\N}$ be a strictly increasing sequence of natural numbers and $\psi\colon (0,\infty)\to(0,\infty)$ be an increasing function. We define a separated net $X(\rho,\mathfrak{l},\psi)$ as follows: Let
	\[
	U_{k}:=[0,l_{k}^{2}]^{d},\qquad k\in\N.
	\]
	and choose arbitrarily a mapping  $g_{\psi}\colon \bigsqcup U_{k}\to \R^{d}$ such that $g_{\psi}$ and the sequence $(U_{k})_{k\in\N}$ satisfy the conditions \eqref{l_i:small_dist} and \eqref{l_i:piecewise_trans} of Lemma~\ref{lemma:miss_inft_many} and additionally $\mb{0}\in g_{\psi}(U_{1})$. Set $R_{k}:=g_{\psi}(U_{k})$ for each $k\in\N$. Next, fix a sequence $(S_{k})_{k\in\N}$ of cubes such that each $S_{k}$ has sidelength $l_{k}$, $S_{k}\subseteq R_{k}$, $\dist(S_{k},\R^{d}\setminus R_{k})\geq \frac{l_{k}^{2}}{4}$ and the vertices of $\partial S_{k}$ belong to the lattice $\frac{1}{2}\Z^{d}\setminus \Z^{d}$. Let $(\Xi_{k})_{k\in\N}$ be the sequence of finite sets $\Xi_{k}\subseteq S_{k}$ given by Lemma~\ref{lemma:Gammak}. Finally, we define the separated net $X(\rho,\mathfrak{l},\psi)$ by
	\[
	X(\rho,\mathfrak{l},\psi):=\bigcup_{k\in\N}\Xi_{k}\cup\left(\Z^{d}\setminus \bigcup_{k\in\N} S_{k}\right).
	\]	
	\end{construction}

\begin{lemma}\label{lemma:bilip_nonequiv}
	Let $\rho\colon [0,1]^{d}\to (0,\infty)$ be a measurable function with $0<\inf\rho\leq\sup\rho<\infty$ and the property that the equation $\Phi_{\sharp}\rho\leb=\leb|_{\Phi([0,1]^{d})}$ has no bilipschitz solutions $\Phi\colon [0,1]^{d}\to\R^{d}$. Let $\mathfrak{l}=(l_{k})_{k\in\N}$ be a strictly increasing sequence of natural numbers. Let $\psi_{1},\psi_{2}\colon (0,\infty)\to(0,\infty)$ be increasing functions such that $\psi_{2}(R+K)\in o(\psi_{1}(R))$ for any fixed $K\in\N$ and $l_k^2\in o(\psi_2(k))$. Then the separated nets
	\[
	X_{i}:=X(\rho,\mathfrak{l},\psi_{i}),\qquad i=1,2,
	\]
	given by Construction~\ref{const_X_psi} are bilipschitz non-equivalent.
\end{lemma}
\begin{proof}
	Assume that $X_{1}$ and $X_{2}$ are BL equivalent and let $f\colon X_{2}\to X_{1}$ be a bijection with
	\[
	L:=\max\set{\lip(f),\lip(f^{-1})}<\infty.
	\] 
	Let the sequences $(U_{k})_{k\in\N}$, $(R_{i,k}:=g_{\psi_{i}}(U_{k}))_{k\in\N}$, $(S_{i,k})_{k\in\N}$ and $(\Xi_{i,k})_{k\in\N}$ and the mapping $g_{i}:=g_{\psi_{i}}\colon \bigsqcup_{k\in\N}U_{k}\to\R^{d}$ be given by Construction~\ref{const_X_psi} with the setting $\psi=\psi_{i}$ for $i=1,2$.
	In particular, this means that
	\begin{equation}\label{eq:Xi}
	X_{i}=\bigcup_{k\in\N}\Xi_{i,k}\cup \left(\Z^{d}\setminus \bigcup _{k\in\N}S_{i,k}\right),\qquad i=1,2.
	\end{equation}
	Observe that the conditions of Lemma~\ref{lemma:miss_inft_many} are satisfied by $\psi_{1}$, $\psi_{2}$, $(U_{k})_{k\in\N}$, $g_{1}$, $g_{2}$, $F:=f^{-1}$ and $D:=X_{1}$. Therefore, by Lemma~\ref{lemma:miss_inft_many}, there is a subsequence $(U_{n_{k}})_{k\in\N}$ of $(U_{n})_{n\in\N}$ such that $f^{-1}\left(X_{1}\cap g_{1}\left(\bigcup_{n\in\N}U_{n}\right)\right)\cap g_{2}(U_{n_{k}})=\emptyset$ for every $k\in\N$. This translates to
	\begin{equation}\label{eq:missing}
	f(X_{2}\cap R_{2,n_{k}})\cap \bigcup_{n\in\N}R_{1,n}=\emptyset\text{ for every $k\in\N$.}
	\end{equation} 
	In what follows it is occasionally necessary to assume that the first index $n_{1}$ of the subsequence $U_{n_{k}}$ is chosen sufficiently large so that, for example, an inequality like $\frac{l_{n_{k}}^{2}}{4}>2l_{n_{k}}$ holds for all $k\in\N$. We will no longer mention this explicitly.
	
	For each $k\in\N$, we set $\widetilde{R}_{k}:=R_{2,n_{k}}$, $\widetilde{S}_{k}:=S_{2,n_{k}}$ and $\Upsilon_{k}:=X_{2}\cap R_{2,n_k}$. Observe that $\Upsilon_{k}\cap \widetilde{S}_{k}=\Xi_{2,n_k}$ and that $\Upsilon_{k}\cap (\widetilde{R}_{k}\setminus \widetilde{S}_{k})=\Z^{d}\cap \widetilde{R}_{k}\setminus \widetilde{S}_{k}$. Moreover, the function $f|_{\bigcup_{k\in\N}\Upsilon_{k}}$ has its image in $\Z^{d}$ due to $\Upsilon_{k}\subseteq\widetilde{R}_{k}=R_{2,n_{k}}$, \eqref{eq:missing}, \eqref{eq:Xi} and $S_{1,n}\subseteq R_{1,n}$. Thus, the only condition of Lemma~\ref{lemma:BK} which is not clearly satisfied by the sequences $\widetilde{R}_{k}$, $\widetilde{S}_{k}$, $\Upsilon_{k}$ and the function $h:=f|_{\bigcup_{k\in\N}\Upsilon_{k}}\colon \bigcup_{k\in\N}\Upsilon_{k}\to \Z^{d}$ is \eqref{eq:buffer}; we verify it shortly. However, first we point out that, once these conditions are verified, applying Lemma~\ref{lemma:BK} in the above setting gives that $h=f|_{\bigcup_{k\in\N}\Upsilon_{k}}$ and therefore also $f$ is not bilipschitz, which is the desired contradiction.

	 It therefore only remains to verify condition~\eqref{eq:buffer} of Lemma~\ref{lemma:BK} for $(\widetilde{R}_{k})_{k\in\N}$, $(\widetilde{S}_{k})_{k\in\N}$, $(\Upsilon_{k})_{k\in\N}$ and the function $h$. Let 
	 \[
	 v\in  B(f(\Upsilon_{k}\cap \widetilde{S}_{k}),\diam \widetilde{S}_{k})\cap \Z^{d}.
	 \]
	  We claim that $v\in X_{1}$. If $v$ is not in $X_{1}$ then, by the definition of $X_{1}$ in Construction~\ref{const_X_psi} and \eqref{eq:Xi}, we must have $v\in \bigcup_{n\in\N}S_{1,n}\subset \bigcup_{n\in\N}R_{1,n}$. Let $b$ denote the net constant of $X_{1}\cap\bigcup_{n\in\N}R_{1,n}$ in $\bigcup_{n\in\N}R_{1,n}$ and choose $v'\in X_{1}\cap \bigcup_{n\in\N}R_{1,n}$ so that $\enorm{v'-v}\leq b$. Let $u'\in X_{2}$ with $f(u')=v'$ and fix a point $w\in \Upsilon_{k}\cap \widetilde{S}_{k}$. Then
	 \begin{multline*}
	 \enorm{u'-w}\leq L\enorm{f(u')-f(w)}\leq L\left(\diam f(\Upsilon_{k}\cap \widetilde{S}_{k})+\diam \widetilde{S}_{k}+b\right)\\
	 \leq 3\sqrt{d}L^{2}l_{n_k}< \frac{l_{n_{k}}^{2}}{4}.
	 \end{multline*}
	 This bound on $\enorm{u'-w}$ together with $w\in \widetilde{S}_{k}$ and $\dist(\widetilde{S}_{k},\R^{d}\setminus \widetilde{R}_{k}))\geq\frac{l_{n_k}^{2}}{4}$ implies that $u'\in X_{2}\cap \widetilde{R}_{k}$. But, according to \eqref{eq:missing}, this in turn requires $v'=f(u')\notin \left(\bigcup_{n\in\N}R_{1,n}\right)$, contrary to the choice of $v'$. We conclude that $v\in X_{1}$.

	 Now, we can choose $z\in X_{2}$ such that $v=f(z)$. Then
	 \begin{multline*}
	 \enorm{z-w}\leq L\enorm{f(z)-f(w)}\leq L\left(\diam f(\Upsilon_{k}\cap \widetilde{S}_{k})+\diam \widetilde{S}_{k}\right)\\
	 \leq L^{2}\sqrt{d}l_{n_k}+L\sqrt{d}l_{n_{k}}< \frac{l_{n_{k}}^{2}}{4}\leq \dist(\widetilde{S}_{k},\R^{d}\setminus \widetilde{R}_{k}).
	\end{multline*}
	 It follows that $z\in X_{2}\cap\widetilde{R}_{k}=\Upsilon_{k}$ and so $v=f(z)\in f(\Upsilon_{k})$. 
\end{proof}

\begin{lemma}\label{lemma:disp_X_psi}
	Let $\rho\colon [0,1]^{d}\to (0,\infty)$ be a measurable function with $0<\inf\rho\leq\sup\rho<\infty$ and $\niceint{[0,1]^{d}}{\rho}{\leb}=1$, $\mathfrak{l}=(l_{k})_{k\in\N}$ be a strictly increasing sequence of natural numbers and $\psi\colon (0,\infty)\to(0,\infty)$ be an increasing function. Let $X_{\psi}:=X(\rho,\mathfrak{l},\psi)$ be the separated net given by Construction~\ref{const_X_psi} and $\phi\colon (0,\infty)\to (0,\infty)$ be an increasing, concave
	function such that $\phi(\psi(k))\in \Omega(l_{k})$. Then $\phi\in \disp_{R}(X_{\psi},\Z^{d})$.
\end{lemma}
\begin{proof}
	The conditions on the sidelength and the location of $S_{k}$ and on the size of $\abs{\Xi_k}$ in Construction~\ref{const_X_psi} and Lemma~\ref{lemma:Gammak}\eqref{prop:crd} ensure that $\abs{X_{\psi}\cap S_{k}}=\abs{\Z^{d}\cap S_{k}}$. Therefore, we may define a bijection $h\colon X_{\psi}\to \Z^{d}$ as follows: on the set $X_{\psi}\setminus \bigcup_{k\in\N}S_{k}$ we define $h$ as the identity. Finally, for each $k\in\N$ we define $h|_{X_{\psi}\cap S_{k}}$ arbitrarily as a bijection $X_{\psi}\cap S_{k}\to \Z^{d}\cap S_{k}$.
	
	The mapping $h$ defined above clearly satisfies 
	\begin{equation}\label{eq:disp_h_first}
	\sup_{x\in \overline{B}(\mb{0},R)}\lnorm{2}{h(x)-x}=\max_{x\in \overline{B}(\mb{0},R)\cap \bigcup_{k\in\N}S_{k}}\lnorm{2}{h(x)-x}\leq\max_{k\colon S_{k}\cap \overline{B}(\mb{0},R)\neq \emptyset}\diam S_{k}.
	\end{equation}
	On the other hand, the conditions of Construction~\ref{const_X_psi}, in particular the properties of the mapping $g_{\psi}$ coming from Lemma~\ref{lemma:miss_inft_many} and $\mb{0}\in g_{\psi}(U_{1})$, ensure that 
	\begin{equation}\label{eq:Sk_inclusion}
	S_{k}\subseteq \R^{d}\setminus B(\mb{0},\psi(k))
	\end{equation}
	for every $k>1$.
	Moreover, given $R>\inf_{x\in S_{2}}\enorm{x}$, there is a maximal $n\in\N$, $n\geq 2$ such that $\inf_{x\in S_{n}}\enorm{x}\leq R$. We infer, using \eqref{eq:Sk_inclusion}, that $R\geq \psi(n)$. This, in combination with \eqref{eq:disp_h_first} implies
	\[
	\frac{\sup_{x\in \overline{B}(0,R)}\lnorm{2}{h(x)-x}}{\phi(R)}\leq \frac{\max_{k\in[n]}\diam S_{k}}{\phi(\psi(n))}= \frac{\sqrt{d}l_{n}}{\phi(\psi(n))}\in O(1).\qedhere
	\]
\end{proof}
Putting together Lemmas~\ref{lemma:bilip_nonequiv} and \ref{lemma:disp_X_psi} it is easy to finish the proof of Theorem~\ref{thm:bil_neq_slow_disp}:
\begin{proof}[Proof of Theorem~\ref{thm:bil_neq_slow_disp}]
		Let $\rho\colon [0,1]^{d}\to (0,\infty)$ be a measurable function with $0<\inf\rho\leq\sup\rho<\infty$ and the property that the equation $\Phi_{\sharp}\rho\leb=\leb|_{\Phi([0,1]^{d})}$ has no bilipschitz solutions $\Phi\colon [0,1]^{d}\to\R^{d}$. Let $\mathfrak{l}=(l_{k})_{k\in\N}$ be a strictly increasing sequence of natural numbers. Let $\Lambda'$ denote the collection of all increasing functions $\psi\colon (0,\infty)\to(0,\infty)$ for which $\phi(\psi(k))\in \Omega(l_{k})$
	and $l_k^2\in o(\psi(k))$. For each $\psi\in \Lambda'$ let $X_{\psi}:=X(\rho,\mathfrak{l},\psi)$ be the separated net of $\R^{d}$ given by Construction~\ref{const_X_psi}. Define an equivalence relation $\sim$ on $\Lambda'$ by $\psi_{1}\sim \psi_{2}$ if $X_{\psi_{1}}$ and $X_{\psi_{2}}$ are BL equivalent. Finally, we may define $\Lambda:=\Lambda'/\sim$. The assertions of the theorem are now readily verified using Lemmas~\ref{lemma:bilip_nonequiv} and \ref{lemma:disp_X_psi}.
\end{proof}
	
	\section{Hierarchy of \texorpdfstring{$\omega$}{\unichar{"1D714}}-regularity of separated nets.}\label{sec:omeg_reg}
Here we prove Theorem~\ref{thm:distinct} and Corollary~\ref{cor:3distinct}. The statements are repeated for the reader's convenience. We also recall from the introduction that for two moduli of continuity $\omega_{1},\omega_{2}$, in the sense of Definition~\ref{def:omeg_reg}, satisfying $\omega_{2}\in o(\omega_{1})$ the notion of $\omega_1$-regularity is formally weaker than $\omega_2$-regularity. That is, the set of $\omega_{2}$-regular separated nets is contained in the set of $\omega_{1}$-separated nets.
\distinct
\begin{proof}
	Define $\phi\colon (0,\infty)\to (0,\infty)$ by $\phi(t)=(\log t)^{\alpha_{0}}$. Then, by Theorem~\ref{thm:bil_neq_slow_disp} there is a separated net $X\subseteq \R^{d}$ which is BL non-equivalent to the integer lattice $\Z^{d}$, but for which $\disp_{R}(X,\Z^{d})\cap O(\phi(R))\neq \emptyset$.
At the same time, \cite[Theorem~1.2 \& Proposition~1.3]{dymond_kaluza2019highly} assert that there are $\omega$-irregular separated nets $Y\subseteq \R^{d}$ and that all such separated nets $Y$ satisfy $\disp_{R}(Y,\Z^{d})\cap O(\phi(R))=\emptyset$.
We conclude that the separated net $X$ must be $\omega$-regular.
\end{proof}
\threedistinct
\begin{proof}
	Let $\omega_{2}(t)=t\left(\log \frac{1}{t}\right)^{\alpha_{0}}$, where $\alpha_{0}=\alpha_{0}(d)$ is given by Theorem~\ref{thm:distinct}, $\omega_{3}(t)=t$ and let $X\subseteq \R^{d}$ be a separated net given by \cite[Theorem~1.2]{dymond_kaluza2019highly}, meaning that $X$ is both $\omega_{2}$- and $\omega_{3}$-irregular. According to \cite[Theorem~5.1]{McM} there exists a H\"older modulus of continuity $\omega_{1}(t)=t^{\beta}$ for some $\beta\in (0,1)$ such that $X$ is $\omega_{1}$-regular. In light of Theorem~\ref{thm:distinct}, it is now clear that the sets of $\omega_i$-regular separated nets in $\R^d$, $i\in[3]$, are pairwise distinct. 
\end{proof}

	\section*{Declarations}
	\paragraph{Funding.} This work was started while both authors were employed at the University of Innsbruck and enjoyed the full support of Austrian Science Fund (FWF): P 30902-N35. It was continued when the first named author was employed at University of Leipzig and the second named author was employed at Institute of Science and Technology of Austria, where he was supported by an IST Fellowship.
	
	
	\bibliographystyle{plain}
	\bibliography{a_main}
\end{document}